\documentclass[12pt]{amsart}
\usepackage[english]{babel}
\usepackage{amsthm}
\usepackage{inputenc}
\usepackage{amssymb}
\usepackage{color}
\usepackage{cite}
\newtheorem{thm}{Theorem}[section]
\newtheorem{cor}[thm]{Corollary}
\newtheorem{lm}[thm]{Lemma}

\newtheorem{defn}[thm]{Definition}
\newtheorem{rem}[thm]{Remark}
\newtheorem{exam}[thm]{Example}
\numberwithin{equation}{section}

\begin{document}

\author[Ayupov]{Shavkat Ayupov}
\email{sh$_{-}$ayupov@mail.ru}
\address{ V.I.Romanovskiy Institute of Mathematics, Uzbekistan Academy of Sciences, 29, Dormon Yoli street, 100125  Tashkent,   Uzbekistan}

\author[Kudaybergenov]{Karimbergen Kudaybergenov}
\email{karim2006@mail.ru}
\address{ Department of Mathematics, Karakalpak State University, Ch. Abdirov 1, Nukus 230113, Uzbekistan}

\author[Omirov]{Bakhrom Omirov}
\email{omirovb@mail.ru}
\address{ National University of Uzbekistan,  4, University street, 100174, Tashkent,   Uzbekistan}

\title[Local and 2-local derivations and automorphisms]{Local and 2-local derivations and automorphisms on simple Leibniz algebras}

\maketitle

\begin{abstract}
The present paper is devoted to local and 2-local derivations
and automorphism of complex finite-dimensional simple Leibniz algebras.
We prove that all  local  derivations and 2-local derivations on a
 finite-dimensional complex simple Leibniz  algebra are automatically
   derivations. We show that  nilpotent  Leibniz  algebras as a rule  admit
     local derivations  and 2-local derivations  which are not derivations.
     Further we consider  automorphisms of simple Leibniz algebras.  We  prove that
every 2-local automorphism on a complex finite-dimensional
 simple Leibniz algebra is an automorphism and show that nilpotent
 Leibniz algebras admit 2-local automorphisms which are not automorphisms.
A similar problem concerning  local automorphism
 on simple Leibniz algebras is reduced to the case of simple Lie algebras.
\end{abstract}

\medskip



\section{Introduction}\label{sec:intro}

Let $\mathcal{A}$ be an associative  algebra. Recall that a linear mapping $\Phi$ of $\mathcal{A}$ into itself is called a local automorphism (respectively, a local derivation) if for every $x\in \mathcal{A}$ there exists an automorphism (respectively, a derivation) $\Phi_x$ of $\mathcal{A},$ depending on $x,$ such that $\Phi_x(x)=\Phi(x).$ These notions were introduced and investigated independently by Kadison~\cite{Kadison90} and Larson and Sourour~\cite{Larson90}. Later, in 1997, P.~\v{S}emrl~\cite{Semrl97}  introduced the concepts of 2-local automorphisms and 2-local derivations. A map $\Phi:\mathcal{A} \rightarrow \mathcal{A}$ (not linear in general) is called a 2-local automorphism (respectively, a 2-local derivation) if for every $x, y\in \mathcal{A},$ there exists an automorphism (respectively, a derivation) $\Phi_{x,y}:\mathcal{A} \rightarrow \mathcal{A}$ (depending on $x, y$) such that $\Phi_{x.y}(x)=\Phi(x),$ $\Phi_{x,y}(y)=\Phi(y).$

The above papers gave rise to series of works devoted to description of mappings which are close to automorphisms and derivations of C*-algebras and  operator algebras. For details we refer to the paper \cite{AK} and the survey \cite{AKR}.

Later, several papers have been devoted to similar notions and corresponding problems for derivations and automorphisms of Lie algebras.

Let $\mathcal{L}$ be a Lie algebra. A derivation (respectively, an automorphism) $\Phi$ of  $\mathcal{L}$ is a linear (respectively, an invertible linear) map $\Phi:\mathcal{L} \rightarrow \mathcal{L}$ which satisfies the condition  $\Phi([x, y]) =[\Phi(x), y] +[x, \Phi(y)]$ (respectively, $\Phi([x, y]) =[\Phi(x), \Phi(y)])$ for all $x, y\in \mathcal{L}.$

The notions of a local derivation (respectively, a local automorphism) and a 2-local derivation (respectively, a 2-local
automorphism) for Lie algebras are defined as above, similar to the associative case. Every derivation (respectively,
automorphism) of a Lie algebra $\mathcal{L}$ is a local derivation (respectively, local automorphism) and a 2-local derivation (respectively, 2-local automorphism). For a given Lie algebra $\mathcal{L},$ the main problem concerning these notions is to prove that they automatically become a  derivation (respectively, an automorphism) or to give examples of local and 2-local derivations or automorphisms of $\mathcal{L},$ which are not derivations or automorphisms, respectively.

Solution of such problems for finite-dimensional Lie algebras over algebraically closed field of
 zero characteristic were obtained in \cite{AKR15, AK2016, AK2016A} and \cite{Chen}.
  Namely, in \cite{AKR15} it is proved that every 2-local derivation on a semi-simple Lie algebra
   $\mathcal{L}$ is a derivation and that each finite-dimensional nilpotent Lie algebra with
    dimension larger than two admits 2-local derivation which is not a derivation.
     In \cite{AK2016} we have proved that every local derivation on semi-simple Lie
      algebras is a derivation and gave examples of nilpotent finite-dimensional Lie
       algebras with local derivations which are not derivations. Concerning 2-local
        automorphism, Chen and Wang in \cite{Chen} prove that if $\mathcal{L}$ is a
        simple Lie algebra of type $\mathcal{A}_l$, $\mathcal{D}_l$ or $\mathcal{E}_k$
        ($k=6,7,8$) over an algebraically closed field of characteristic zero, then
        every 2-local automorphism of $\mathcal{L}$ is an automorphism. Finally,
        in \cite{AK2016A} Ayupov and Kudaybergenov generalized this result of \cite{Chen}
        and proved that every 2-local automorphism of a  finite-dimensional semi-simple
         Lie algebra over an algebraically closed field of characteristic zero is an
         automorphism. Moreover, they show also that every  nilpotent Lie algebra with
         finite dimension larger than two admits 2-local automorphisms which is not an automorphism.
It should be noted that similar problems for local automorphism of finite-dimensional Lie algebras still remain open.

Leibniz algebras present a "non antisymmetric" extension of Lie algebras.
In last decades a series of papers have been devoted to the structure theory and classification
of finite-dimensional Leibniz algebras.
 Several classical theorems from Lie algebras theory have been extended
to the Leibniz algebras case. For some details from the theory of
Leibniz algebras we refer to the papers \cite{AAO05, AAO06, AAOK,
L93, Fial, LP}. In particular, for a finite-dimensional simple Leibniz
algebras over an algebraically closed field of characteristic
zero, derivations have been completely described in \cite{RMO}.

In the present paper we study local and 2-local derivations and automorphisms of finite-dimensional simple complex Leibniz algebras.

In Section 2 we give some preliminaries from the Leibniz algebras
theory. In Section~3 we prove that every 2-local derivation on a
simple Leibniz algebra $\mathcal{L}$ is a derivation.  We also
prove that all nilpotent Leibniz algebras (except so-called
null-filiform Leibniz algebra) admit 2-local derivations which are
not derivations. Similar results for local derivations on simple
Leibniz algebras are obtained in Section~4. Namely, we show that
every local derivation on a simple complex Leibniz algebra is a
derivation and that each finite-dimensional filiform, Leibniz
algebra $\mathcal{L}$ with $\dim \mathcal{L}\geq 3$ admits a local
derivation which is not a derivation.

In Section~5  we study automorphisms of simple Leibniz algebras.

Finally, in Section 6 we consider 2-local and local automorphisms
of finite-dimensional Leibniz algebras. First we show that every
2-local automorphism of a  complex simple Leibniz algebra is an
automorphism and prove that each $n$-dimensional nilpotent Leibniz
algebra such that $n\geq2$ and $\dim[\mathcal{L}, \mathcal{L}]\leq
n-2$, admits a 2-local automorphism which is not an automorphism.
At the end of this Section 6 we show that the problem concerning
local automorphisms of simple complex Leibniz algebras is reduced
to the similar problem for simple Lie algebras, which is,
unfortunately, still open.

\section{Preliminaries}

In this section we give some necessary definitions and preliminary
results.

\begin{defn} An algebra $(\mathcal{L},[\cdot,\cdot])$ over a field $\mathbb{F}$ is called a
Leibniz algebra if it is defined by the identity
$$[x,[y,z]]=[[x,y],z] - [[x,z],y], \ \mbox{for all}\ x,y \in \mathcal{L},$$
which is called Leibniz identity.
\end{defn}

It is a generalization of the Jacobi identity since under the condition of
anti-symmetricity of the product ``[$\cdot,\cdot$]'' this identity
changes to the Jacobi identity. In fact, the definition above
is the notion of the right Leibniz algebra, where ``right''
indicates that any right multiplication operator is a derivation
of the algebra. In the present  paper the term ``Leibniz algebra'' will
always mean the ``right Leibniz algebra''. The left Leibniz
algebra is characterized by the property that any left
multiplication operator is a derivation.

Let $\mathcal{L}$ be a Leibniz algebra and $\mathcal{I}$ be the
ideal generated by squares in $\mathcal{L}:$
$\mathcal{I}=\textrm{id}\langle[x,x]\  | \ x\in \mathcal{L}
\rangle.$ The quotient $\mathcal{L}/\mathcal{I}$ is called  "the
associated Lie algebra" of the Leibniz algebra $\mathcal{L}.$ The
natural epimorphism $\varphi  : \mathcal{L} \rightarrow
\mathcal{L}/\mathcal{I}$ is a homomorphism of Leibniz algebras.
The ideal $\mathcal{I}$ is the minimal ideal with  the
property that the quotient algebra is a Lie algebra. It is easy to
see that the ideal $\mathcal{I}$ coincides with the subspace of
$\mathcal{L}$ spanned by the squares,  and that  $\mathcal{L}$ is the left
annihilator of $\mathcal{I},$ i.e., $[\mathcal{L},\mathcal{I}]=0.$

\begin{defn}
A Leibniz algebra $\mathcal{L}$ is called simple if its ideals are
only $\{0\}, \mathcal{I}, \mathcal{L}$ and
$[\mathcal{L},\mathcal{L}]\neq \mathcal{I}.$
\end{defn}
This definition agrees with that of simple Lie algebra, where
$\mathcal{I}=\{0\}.$

For a given Leibniz algebra $\mathcal{L}$ we define derived
sequence as follows:
$$
\mathcal{L}^{[1]}=\mathcal{L},\quad \mathcal{L}^{[s+1]}=[\mathcal{L}^{[s]},\mathcal{L}^{[s]}], \quad s \geq 1.
$$

\begin{defn} A Leibniz algebra $\mathcal{L}$ is called solvable, if there exists  $s\in\mathbb N$ such that $\mathcal{L}^{[s]}=\{0\}.$
\end{defn}

It is known that sum of solvable ideals of a Leibniz algebra is
solvable ideal too. Therefore each  Leibniz algebra contains a maximal solvable ideal which is called {\it
solvable radical.}

The following theorem recently proved by D.~Barnes \cite{Barnes}
presents an analogue of Levi--Malcev's theorem for Leibniz
algebras.

\begin{thm} \label{t2}  \label{thmBarnes} Let $\mathcal{L}$ be a finite dimensional Leibniz algebra over a
field of characteristic zero and let $\mathcal{R}$ be its solvable radical. Then there exists a semisimple Lie subalgebra
$\mathcal{S}$ of $\mathcal{L}$ such that $\mathcal{L}=\mathcal{S}\dot{+}\mathcal{R}.$
\end{thm}

This theorem applied to a simple Leibniz algebra $\mathcal{L}$ gives

\begin{cor}
Let $\mathcal{L}$ be a simple Leibniz algebra over a field of
characteristic zero and let $\mathcal{I}$ be the ideal generated by
squares in $\mathcal{L},$ then there exists a simple Lie algebra
$\mathcal{G}$ such that $\mathcal{I}$ is an irreducible module
over $\mathcal{G}$ and
$\mathcal{L}=\mathcal{G}\dot{+}\mathcal{I}.$
\end{cor}

Further we shall use the following important result \cite{Hum}.

\begin{thm} (Schur's Lemma) \label{SL} Let $\mathcal{G}$ be a complex Lie algebra and let $\mathcal{U}$
and $\mathcal{V}$ be irreducible $\mathcal{G}$-modules. Then
\begin{itemize}
\item[(i)] Any $\mathcal{G}$-module homomorphism $\Theta : \mathcal{U} \rightarrow \mathcal{V}$
is either trivial or an isomorphism;
\item[(ii)] A linear map $\Theta : \mathcal{V} \rightarrow \mathcal{V}$  is a $\mathcal{G}$-module homomorphism if and only
if $\Theta = \lambda id_{|_\mathcal{V}}$ for some $\lambda \in
\mathbb{C}.$
\end{itemize}
\end{thm}

The notion  of derivation for a Leibniz algebras  is
defined similar to the the Lie algebras case as follows.
\begin{defn}
A linear transformation $d$ of a Leibniz algebra $\mathcal{L}$ is
said to be a derivation if for any $x, y\in \mathcal{L}$ one has
$$
D([x,y])=[D(x),y]+[x, D(y)].$$
\end{defn}

Let $a$ be an element of a Leibniz algebra $\mathcal{L}.$ Consider
the operator of right multiplication $R_a:\mathcal{L}\to
\mathcal{L}$, defined by $R_a(x)=[x,a].$  The Leibniz identity which characterizes  Leibniz algebras  exactly means  that every right multiplication operator $R_a$ is a derivation. Such derivations are called  inner derivation on  $\mathcal{L}$ . Denote by $Der (\mathcal{L})$ - the space of all derivations of $\mathcal{L}$.

Now we shall present the main subjects considered in this paper, so-called
local and $2$-local derivation.

\begin{defn}
A linear operator  $\Delta : \ \mathcal{L} \ \rightarrow
\mathcal{L}$ is called a local derivation if for any $x\in
\mathcal{L}$ there exists a derivation $D_{x}\in Der
(\mathcal{L})$ such that
$$
\Delta(x)=D_{x}(x).
$$
\end{defn}

\begin{defn}
A map $\Delta :  \mathcal{L} \rightarrow \mathcal{L}$ (not
necessary  linear) is called $2$-local derivation if for any $x,
y\in \mathcal{L}$ there exists a derivation $D_{x,y}\in Der
(\mathcal{L})$ such that
$$
\Delta(x)=D_{x,y}(x), \quad \Delta(y)=D_{x,y}(y).
$$
\end{defn}

From now on we assume that all algebras are considered  over the field of complex numbers $\mathbb{C}$ and suppose that $\mathcal{L}$ is a  non-Lie Leibniz algebra, i.e. $\mathcal{I}\neq\{0\}.$

Now we give a description of  derivations on simple Leibniz algebras obtained in \cite{RMO}.

Let \(\mathcal{L}\) be a simple Leibniz algebra with
$\mathcal{L}=\mathcal{G}\dot{+}\mathcal{I}.$ Consider a projection
operator \(\textrm{pr}_{\mathcal{I}}\) from \(\mathcal{L}\) onto
\(\mathcal{I},\) that is
\begin{equation}\label{proj}
\textrm{pr}_\mathcal{I}(x+i)=i,\, x+i\in
\mathcal{L}\dot{+}\mathcal{I}.
\end{equation}

Suppose that   \(\mathcal{G}\) and \(\mathcal{I}\) are not
isomorphic as \(\mathcal{G}\)-modules. Then any derivation $D$ on
$\mathcal{L}$ can be represented as
\begin{equation}\label{neq}
D=R_a+\lambda \textrm{pr}_\mathcal{I},
\end{equation}
where \(R_a\) is an inner derivation generated by an element
\(a\in \mathcal{G},\) \(\textrm{pr}_\mathcal{I}\) is a derivation
of the form~\eqref{proj}, \(\lambda\in \mathbb{C}.\)

Now let us assume that \(\mathcal{G}\) and \(\mathcal{I}\) are
isomorphic as \(\mathcal{G}\)-modules. There exists a unique (up
to multiplication by constant)  isomorphism \(\theta\) of linear
spaces \(\mathcal{G}\) and \(\mathcal{I}\) such that
\(\theta([x,y])=[\theta(x), y]\) for all \(x, y\in \mathcal{G},\)
i.e., \(\theta\) is a module isomorphism of
\(\mathcal{G}\)-modules  \(\mathcal{G}\) and \(\mathcal{I}.\) Let
us extend \(\theta\) onto \(\mathcal{L}\) as
\begin{equation}\label{theta}
\theta(x+i)=\theta(x),\, x+i\in \mathcal{L}\dot{+}\mathcal{I}.
\end{equation}

For a simple Leibniz algebra \(\mathcal{L}\) with \(\dim
\mathcal{G}=\dim \mathcal{I}\)  any derivation $D$ on
$\mathcal{L}$ can be represented as
\begin{equation}\label{equal}
D=R_a+\omega\theta+\lambda \textrm{pr}_\mathcal{I},
\end{equation}
where \(a\in \mathcal{G},\) \(\textrm{pr}_\mathcal{I}\) is a
derivation of the form~\eqref{proj} and \(\theta\) is a derivation
of the form~\eqref{theta}, \(\lambda, \,\omega\in \mathbb{C}.\)

\section{2-Local derivations on  Leibniz algebras}
\label{thr}

\subsection{2-Local derivations on simple Leibniz algebras}

\

    The first main result of this section is the following.

\begin{thm}\label{thm2local}
Let $\mathcal{L}$ be a simple complex Leibniz algebra. Then any
2-local derivation on  $\mathcal{L}$ is  a derivation.
\end{thm}

For the proof of this Theorem we need several Lemmata.

From theory of representation of semisimple Lie algebras \cite{Hum} we have that a Cartan subalgebra $\mathcal{H}$ of Lie algebra $\mathcal{G}$ acts diagonalizable on  $\mathcal{G}$-module $\mathcal{I}:$
\[
\mathcal{I}=\bigoplus_{\alpha\in \Gamma} \mathcal{I}_\alpha,
\]
where
\begin{eqnarray*}
\mathcal{I}_\alpha & = & \{i\in \mathcal{I}:  [i, h]=\alpha(h)i,
\,\, \forall \, h\in \mathcal{H}\},\\
\Gamma & = & \{\alpha \in \mathcal{H}^\ast: \mathcal{I}_\alpha\neq
\{0\}\}
\end{eqnarray*}
 and $\mathcal{H}^\ast$ is the space of all linear functionals on $\mathcal{H}.$ Elements of $\Gamma$ are called
  weights of~$\mathcal{I}.$

For every \(\beta\in \Gamma\) take a non zero element
\(i^{(0)}_\beta\in \mathcal{I}_\beta.\) Set
\begin{eqnarray*}
i_0=\sum\limits_{\beta\in \Gamma}i^{(0)}_\beta.
\end{eqnarray*}

Fix  a regular element \(h_0\) in \(\mathcal{H},\) in particular
\[
 \{x\in \mathcal{G}; [x, h_0]=0\}=\mathcal{H}.
\]
In the following two Lemmata~\ref{first}--\ref{fourth} we assume that \(\mathcal{L}=\mathcal{G}\dot{+}\mathcal{I}\) is  a
Leibniz algebra  such that \(\mathcal{G}\) and \(\mathcal{I}\) are
non isomorphic as \(\mathcal{G}\)-modules.

\begin{lm}\label{first}
Let  \(D\) be a derivation on \(\mathcal{L}\) such that
\(D(h_0+i_0)=0.\) Then $D|_{\mathcal{I}}\equiv 0.\)
\end{lm}

\begin{proof}
By \eqref{neq} there exist an element \(a\in \mathcal{G}\) and a
number \(\lambda\in \mathbb{C}\) such that
\(D=R_a+\lambda\textrm{pr}_\mathcal{I}.\) We have
\[
0=D(h_0+i_0)=[h_0+i_0, a]+\lambda i_0=[h_0, a]+[i_0, a]+\lambda
i_0.
\]
Since \([h_0, a]\in \mathcal{G}\) and \([i_0, a]+\lambda i_0\in
\mathcal{I},\) it follows that \([h_0, a]=0\) and \([i_0,
a]+\lambda i_0=0.\) Since \(h_0\) is a regular element, we have
that \(a\in \mathcal{H}.\) Further
\[
0=[i_0, a]+\lambda i_0=\left[\sum\limits_{\beta\in
\Gamma}i^{(0)}_\beta, a\right]+\lambda i_0=\sum\limits_{\beta\in
\Gamma}\beta(a)i^{(0)}_\beta+\lambda\sum\limits_{\beta\in
\Gamma}i^{(0)}_\beta.
\]
Thus \(\beta(a)=-\lambda\) for all \(\beta\in \Gamma.\)

Let $i$ be an arbitrary element of $I$, then it has a decomposition \(i=\sum\limits_{\beta\in \Gamma}i_\beta,\) where \(i_\beta\in
I_\beta,\, \beta\in \Gamma.\) From \(\beta(a)=-\lambda\) for all \(\beta\in \Gamma\)
we get
\begin{eqnarray*}
D(i)& = & [i, a]+\lambda i=\left[\sum\limits_{\beta\in
\Gamma}i_\beta, a\right]+\lambda\sum\limits_{\beta\in
\Gamma}i_\beta=\\
& = & \sum\limits_{\beta\in \Gamma}\beta(a)
i_\beta+\lambda\sum\limits_{\beta\in \Gamma}i_\beta=0.
\end{eqnarray*}
The proof is complete.
\end{proof}

\begin{lm}\label{fourth}
Let $D$ be a derivation on \(\mathcal{L}\) such  that
\(D(h_0+i_0)=0.\) Then \(D=0.\)
\end{lm}

\begin{proof}  Let \(y\in \mathcal{L}\) be an arbitrary element. Since \([y, y]\in \mathcal{I},\) Lemma~\ref{first}
implies that  \(D([y, y])=0.\) The derivation  identity
\[
D([y, y])=[D(y), y]+[y, D(y)]
\]
implies that
\begin{equation}
\label{square} [D(y), y]+[y, D(y)]=0.
\end{equation}
Putting    \(y=x+i\in \mathcal{G}+\mathcal{I}\) in  \eqref{square}
we obtain that
\[
[D(x+i), x+i]+[x+i, D(x+i)]=0.
\]
Taking into account   \(D(x+i)=D(x)\) we have that
\[
[D(x), x]+[x, D(x)]+[i, D(x)]=0.
\]
Using \eqref{square} from  the last equality  we have
\([\mathcal{I}, D(x)]=0.\) Further for arbitrary elements \(z\in
\mathcal{G}\) and \(i\in \mathcal{I}\) we have
\begin{eqnarray*}
[i, [D(x), z]] & = & [[i, D(x)], z]-[[i, z], D(x)]=0
\end{eqnarray*}
and
\begin{eqnarray*}
[i, [z, D(x)]] & = & [[i, z], D(x)]-[[i, D(x)], z]=0.
\end{eqnarray*}
This means that  \([\mathcal{I}, \mathcal{G}_{D(x)}]=0,\) where
\(\mathcal{G}_{D(x)}\) is an ideal in \(\mathcal{G}\) generated by
the element \(D(x).\) Since \(\mathcal{I}\) is an irreducible
module over the simple Lie algebra \(\mathcal{G},\) we obtain that
\(\mathcal{G}_{D(x)}=\{0\},\) i.e., \(D(x)=0.\) Finally,
\(D(y)=D(x+i)=D(x)=0.\) The proof is complete.
\end{proof}

Consider a  decomposition for $\mathcal{G},$ called the
\textit{root decomposition}
$$
\mathcal{G}=\mathcal{H}\oplus \bigoplus\limits_{\alpha\in
\Phi}\mathcal{G}_\alpha,
$$
where
\begin{eqnarray*}
\mathcal{G}_\alpha & = & \{x\in \mathcal{G}:  [h, x] =
\alpha(h)x,\, \forall\,  h \in
\mathcal{H}\},\\
\Phi & = & \{\alpha \in \mathcal{H}^\ast\setminus \{0\}:
\mathcal{G}_\alpha\neq \{0\}\}.
\end{eqnarray*}
The set $\Phi$ is called the \textit{root system} of
$\mathcal{G},$ and subspaces $\mathcal{G}_\alpha$ are called the
\textit{root subspaces}.

Further if  \(\mathcal{G}\) and \(\mathcal{I}\) are  isomorphic
\(\mathcal{G}\)-modules, then  the decomposition for $\mathcal{I}$
can be written as
$$
\mathcal{I}= \mathcal{I}_0\oplus \bigoplus\limits_{\alpha\in
\Phi}\mathcal{I}_\alpha,
$$
where \(\mathcal{I}_0=\theta(\mathcal{H})\) and
\(\mathcal{I}_\beta=\theta(\mathcal{G}_\beta)\) for all \(\beta\in
\Phi.\) This follows from
\[
[\theta(x_\beta), h]=\theta([x_\beta, h])=\beta(h)\theta(x_\beta)
\]
and
\[
[\theta(h'), h]=\theta([h', h])=0
\]
for all \(h, h'\in \mathcal{H}\) and \(x_\beta\in
\mathcal{G}_\beta,\) \(\beta\in \Phi.\)

For every \(\beta\in \Phi\) take a non zero element \(x^{(0)}_\beta\in \mathcal{G}_\beta\) and put
\begin{center}
\(x_0=\sum\limits_{\beta\in \Gamma}x^{(0)}_\beta\) and
\(i_0=\theta(x_0).\)
\end{center}
It is clear that \([x^{(0)}_\beta, h]=\beta(h)x^{(0)}_\beta\) for
all \(h\in \mathcal{H}\) and \(\beta\in \Phi.\)

By \cite[Lemma~2.2]{Wang}, there exists an element $h_0\in
\mathcal{H}$ such that $\alpha(h_0) \neq \beta(h_0)$ for every
$\alpha, \beta\in \Phi, \alpha\neq \beta.$ In particular,
$\alpha(h_0) \neq 0$ for every $\alpha\in \Phi.$ Such elements
$h_0$ are called \textit{strongly regular} elements of
$\mathcal{G}.$ Again by \cite[Lemma~2.2]{Wang}, every strongly
regular element \(h_0\) is a \textit{regular element}, i.e.
$$
\{x\in \mathcal{G}: [h_0, x]=0\}=\mathcal{H}.
$$
Choose a fixed strongly regular element $h_0\in \mathcal{H}.$

\begin{lm}\label{sixth}
 Suppose that \(a, h\in
\mathcal{H},\) \(h\neq 0\) and $\lambda, \omega\in \mathbb{C}$ are
such that
\[
\omega\theta(h)+[i_0, a]+\lambda i_0=0.
\]
Then \(a=0\) and \(\lambda=\omega=0.\)
\end{lm}

\begin{proof}
We have
\begin{eqnarray*}
0 & = & [i_0, a]+\lambda i_0+\omega \theta(h)=[\theta(x_0),
a]+\lambda \theta(x_0)+\omega \theta(h)=\\
& = & \theta([x_0, a])+\lambda \theta(x_0)+\omega
\theta(h)=\theta([x_0, a]+\lambda x_0+\omega h).
\end{eqnarray*}
Since \(\theta\) is an isomorphism of linear spaces
\(\mathcal{G}\) and \(\mathcal{I},\) it follows that \([x_0,
a]+\lambda x_0+\omega h=0.\) Further
\begin{eqnarray*}
0 & = & [x_0, a]+\lambda x_0+\omega h=\left[\sum\limits_{\beta\in
\Phi}x^{(0)}_\beta, a\right]+\lambda
x_0+\omega h=\\
& = & \sum\limits_{\beta\in
\Phi}\beta(a)x^{(0)}_\beta+\lambda\sum\limits_{\beta\in
\Phi}x^{(0)}_\beta+\omega h.
\end{eqnarray*}
Now we multiply both sides of this equality by the regular element $h_0\in \mathcal{H},$ then we get
$$
0=\sum\limits_{\beta\in
\Phi}\beta(a)\beta(h_0)x^{(0)}_\beta+\lambda\sum\limits_{\beta\in
\Phi}\beta(h_0)x^{(0)}_\beta.
$$
From this we obtain $\beta(a)\beta(h_0)=-\lambda\beta(h_0)$ for
all \(\beta\in \Phi.\) Since \(h_0\) is a strongly regular element
it follows that  $\beta(h_0)\neq 0$ for all  \(\beta\in \Phi,\)
then we derive that $\beta(a)=-\lambda.$ Putting in the last
equality the root \(-\beta\) instead of \(\beta\) we obtain \(\beta(a)=\lambda.\)  Thus, \(\lambda=0\) and \(\beta(a)=0\) for
all \(\beta\in \Phi.\) Since the set \(\Phi\) contains \(k=\dim
\mathcal{H}\) linearly independent elements, it follows  that
\(\Phi\) separates points of \(\mathcal{H},\) and therefore  we
get \(a=0.\) Further from
\[
\omega\theta(h)+[i_0, a]+\lambda i_0=0
\]
we obtain that \(\omega\theta(h)=0.\) Since \(\theta(h)\) is non zero, it
follows that \(\omega=0.\) The proof is complete.
\end{proof}

\begin{lm}\label{seven}
Let $D$ be a derivation on \(\mathcal{L}\) such  that
\(D(h_0+i_0)=0.\) Then \(D=0.\)
\end{lm}

\begin{proof} By~\eqref{equal} there exist an
element \(a\in \mathcal{G}\) and  numbers \(\lambda, \theta\in
\mathbb{C}\) such that
\[
D=R_a+\omega \theta+\lambda\textrm{pr}_\mathcal{I}.
\]
We have
\[
0=D(h_0+i_0)=[h_0+i_0, a]+\omega\theta(h_0)+\lambda i_0=[h_0,
a]+\omega\theta(h_0)+[i_0, a]+\lambda i_0.
\]
Since \([h_0, a]\in \mathcal{G}\) and \(\omega\theta(h_0)+[i_0,
a]+\lambda i_0\in \mathcal{I},\) it follows that \([h_0, a]=0\)
and \(\omega\theta(h_0)+[i_0, a]+\lambda i_0=0.\) Since \(h_0\) is
a strongly regular element, we have that \(a\in \mathcal{H}.\) Now
Lemma~\ref{sixth} implies that \(a=0\) and \(\lambda=\omega=0.\)
This means that \(D=0.\) The proof is complete.
\end{proof}

Now we are in position to prove Theorem~\ref{thm2local}.

\textit{Proof of Theorem~\ref{thm2local}}. Let \(\Delta\) be a
2-local derivation on \(\mathcal{L}.\)  Take a derivation \(D\) on
\(\mathcal{L}\) such that
\[
\Delta(h_0+i_0)=D(h_0+i_0).
\]
Consider the 2-local derivation \(\Delta-D.\) Let \(x\in
\mathcal{L}.\)  Take a derivation \(\delta\) on \(\mathcal{L}\)
such that
\[
(\Delta-D)(h_0+i_0)=\delta(h_0+i_0),\,\, (\Delta-D)(x)=\delta(x).
\]
Since \(\delta(h_0+i_0)=0,\) by Lemmata~\ref{fourth} and
\ref{seven} we obtain that \(\delta(x) =0,\) i.e.,
\(\Delta(x)=D(x).\) This means that \(\Delta=D\) is a derivation.
 The proof is complete. $\Box$

\begin{rem} The analogues  of Lemmata~\ref{fourth} and~\ref{seven} are not true for
simple Lie algebras.

Let \(\mathcal{G}\) be a simple Lie algebra and suppose that there
exists a non zero element \(x_0\in\mathcal{G}\) such that
\[
D(x_0)=0,\, D\in \textrm{Der}(\mathcal{G})\,\Rightarrow \, D=0.
\]
Take the inner derivation \(D=R_{x_0}\) generated by the element
\(x_0.\) Then \(D(x_0)=0,\) but \(D\) is a non trivial derivation.
\end{rem}

\subsection{$2$-local derivations of nilpotent Leibniz algebras}

\

In this subsection under a certain assumption  we give a general  construction  of $2$-local derivations which are not derivations for an  arbitrary variety (not necessarily associative,  Lie or Leibniz)  of  algebras. This construction then applied to show that nilpotent Leibniz algebras always admit 2-local derivations which are not derivations.

For an arbitrary  algebra $\mathcal{L}$ with multiplication denoted as $xy$ let
\[
\mathcal{L}^2=\textrm{span}\{xy: x, y \in \mathcal{L}\}
\]
and
\[
\textrm{Ann}(\mathcal{L})=\{x\in \mathcal{L}:
xy=yx=0\,\,\textrm{for all}\,\, y\in \mathcal{L}\}.
\]
Note  that a linear operator $\delta$ on $\mathcal{L}$ such that
$\delta|_{\mathcal{L}^2}\equiv 0$ and
$\delta(\mathcal{L})\subseteq \textrm{Ann}(\mathcal{L})$ is a
derivation. Indeed, for every $x,y\in \mathcal{L}$ we have
$$
\delta(xy)=0=\delta(x)y+x\delta(y).
$$

\begin{thm}\label{pure}
Let $\mathcal{L}$ be a $n$-dimensional  algebra with $n\geq2.$
Suppose that
\begin{itemize}
\item[(i)] $\dim\mathcal{L}^2\leq n-2;$
\item[(ii)] the annihilator
 $\textrm{Ann}(\mathcal{L})$ of $\mathcal{L}$ is non trivial.
\end{itemize}
Then $\mathcal{L}$ admits a $2$-local derivation which is not a
derivation.
\end{thm}

\begin{proof}  Let us consider a decomposition of  $\mathcal{L}$ in the following
form
$$
\mathcal{L}=\mathcal{L}^2\oplus V.
$$
Due to $\dim\mathcal{L}^2\leq n-2,$ we have $\dim V=k\geq 2.$ Let
$\{e_1,\ldots, e_k\}$ be a basis of $V.$

 Let us define a homogeneous non
additive function $f$ on $\mathbb{C}^2$ as follows
\begin{displaymath}
f(y_1, y_2) = \left\{ \begin{array}{ll}
\frac{\textstyle y_1^2}{\textstyle  y_2}, & \textrm{if\,\, $y_2\neq 0$}\\
0, & \textrm{if\,\, $y_2= 0,$}
\end{array} \right.
\end{displaymath}
where $(y_1, y_2)\in \mathbb{C}.$

Let us fix a non zero element  $z\in \textrm{Ann}(\mathcal{L}).$
Define an operator $\Delta$ on $\mathcal{L}$ by
$$
\Delta(x)=f(\lambda_1,\lambda_2)z,\, \
x=x_1+\sum\limits_{i=1}^k\lambda_i e_i,
$$
where $\lambda_i\in \mathbb{C},$ $i=1,\ldots,k,$ $x_1\in
\mathcal{L}^2.$ The operator $\Delta$ is not a derivation since it
is not linear.

Let us now show that $\Delta$ is a $2$-local derivation.  Define a
linear operator $\delta$ on $\mathcal{L}$ by
$$
\delta(x)=(a\lambda_1+b\lambda_2)z,\, \ x=x_1+\sum\limits_{i=1}^k
\lambda_i e_i,
$$
where $a, b\in \mathbb{C}.$ Since $\delta|_{\mathcal{L}^2}\equiv
0$ and $\delta(\mathcal{L})\subseteq \textrm{Ann}(\mathcal{L})$
the operator $\delta$ is a derivation.

Let $x=x_1+\sum\limits_{i=1}^k \lambda_i e_i$ and
$y=y_1+\sum\limits_{i=1}^k \mu_i e_i$ be elements of
$\mathcal{L}.$ We choose $a$ and $b$ such that
$$
\Delta(x)=\delta(x),\,\, \Delta(y)=\delta(y).
$$
Let us rewrite the above equalities  as system of linear equations
with respect to unknowns $a, b$ as follows
\begin{displaymath}
\left\{ \begin{array}{ll}
\lambda_1 a+\lambda_2 b & = f(\lambda_1,\lambda_2) \\
\mu_1 a+\mu_2 b & = f(\mu_1,\mu_2)
\end{array} \right.
\end{displaymath}
Since the function $f$ is homogeneous the system has non trivial
solution. Therefore, $\Delta$ is a $2$-local derivation, as
required. The proof is complete.
\end{proof}

Given a  Leibniz algebra $\mathcal{L}$, we define the
\emph{lower central sequence} defined recursively as
$$\mathcal{L}^1=\mathcal{L}, \ \mathcal{L}^{k+1}=[\mathcal{L}^k,\mathcal{L}],   \ \ k \geq 1.$$

\begin{defn} A Leibniz algebra $\mathcal{L}$ is said to be nilpotent, if there exists $t\in\mathbb N$ such that $\mathcal{L}^{t}=\{0\}$. The minimal number $t$ with such property is said to be the index of nilpotency of the algebra $\mathcal{L}$.
\end{defn}

Since for an nilpotent algebra we have $\dim \mathcal{L}^2\leq n-1$, the index of nilpotency of an
$n$-dimensional nilpotent Leibniz algebra is not greater than
$n+1$.

A Leibniz algebra $\mathcal{L}$ is called \textit{null-filiform}
if $\dim \mathcal{L}^k=n+1-k$ for $1\leq k \leq n+1.$

Clearly, a null-filiform Leibniz algebra has maximal index of nilpotency. Moreover, it is easy to show that a nilpotent Leibniz algebra is null-filiform if and only if it is a one-generated algebra (see  \cite{AO01}) . Note that this notion has no sense in the  Lie algebras case, because Lie algebras  are at least two-generated.

For every nilpotent Leibniz algebra with
nilindex equal to $t$  we have that
$\{0\}\neq\mathcal{L}^{t-1}\subseteq \textrm{Ann}(\mathcal{L}).$ It is known  \cite{AO01} that up to isomorphism there exists a unique $n$-dimensional nilpotent Leibniz which satisfies the  condition $\dim \mathcal{L}^2=n-1$ which  is the null-filiform algebra, i.e. $\dim \mathcal{L}^2\leq n-2$ for all nilpotent Leibniz algebras except the null-filiform. Therefore Theorem~\ref{pure} implies the following result.

\begin{cor}\label{nil}
Let $\mathcal{L}$ be a finite-dimensional non null-filiform nilpotent   Leibniz algebra    with $\dim \mathcal{L}\geq 2.$ Then $\mathcal{L}$ admits a $2$-local derivation which is not a derivation.
\end{cor}

\begin{rem} Let us show that  every  $2$-local derivation of the algebra $NF_n$ is a derivation.
\end{rem}

It is know that the unique null-filiform algebra, denoted by $NF_n$, admits a basis $\{e_1, e_2, \dots, e_n\}$ in which its the table of multiplications is the  following (see  \cite{AO01}):
$$
[e_i,e_1]=e_{i+1}, \ 1\leq i \leq n.
$$

Let $D\in Der(NF_n)$ and $D(e_1)=\sum\limits_{i=1}^{n}\alpha_i
e_i$ for some $\alpha_i\in \mathbb{C}$, then it is easy to check
that
\begin{eqnarray}\label{dernf}
D(e_j)=j \alpha_1 e_j +\sum_{i=2}^{n-j+1}\alpha_ie_{j+i-1}, \
2\leq j \leq n.
\end{eqnarray}

Let $\Delta$ be a 2-local derivation. For the elements $e_1, x,
y\in NF_n$  there exist derivations  $D_{e_1, x}$ and
$D_{e_1, y}$ such that
$$
\Delta(e_1)=D_{e_1, x}(e_1)=D_{e_1,
y}(e_1)=\sum\limits_{i=1}^{n}\alpha_i e_i.
$$
From \eqref{dernf} we conclude that each derivation  on $NF_n$ is
uniquely defined by its value on the element~$e_1.$ Therefore, $D_{e_1,
x}(z)=D_{e_1, y}(z)$ for any $z\in NF_n$. Thus, we obtain that if
$\Delta(e_1)=D_{e_1, x}(e_1)$ for some $D_{e_1, x}$, then
$\Delta(z)=D_{e_1, x}(z)$ for any $z\in NF_n,$ i.e., $\Delta$ is a
derivation.

\section{Local derivations on  Leibniz algebras}
\label{four}

\subsection{Local derivations on simple Leibniz algebras}

\

Now we shall give the main result concerning local derivations on
simple Leibniz algebras.

\begin{thm}\label{thmlocal}
Let $\mathcal{L}$ be a simple complex Leibniz algebra. Then any
local derivation on  $\mathcal{L}$ is  a derivation.
\end{thm}

Let us first consider a simple Leibniz algebra
\(\mathcal{L}=\mathcal{G}\dot{+}\mathcal{I}\) such that
\(\mathcal{G}\) and \(\mathcal{I}\) are  isomorphic
\(\mathcal{G}\)-modules.

\begin{lm}\label{GeqI}
Let \(\Delta\) be a local derivation  on \(\mathcal{L}\) such that
\(\Delta\) maps \(\mathcal{L}\) into \(\mathcal{I}.\)  Then
\(\Delta\)  is  a derivation.
\end{lm}

\begin{proof}

Fix a basis \(\{x_1, \ldots, x_m\}\) in \(\mathcal{G}.\) In this
case the system of vectors \(\{y_i: y_i=\theta(x_i), i\in
\overline{1, m}\}\) is a basis in \(\mathcal{I},\) where
\(\theta\) is a module isomorphism of \(\mathcal{G}\)-modules
\(\mathcal{G}\) and~\(\mathcal{I},\) in particular,  \(\theta([x,
y])=[\theta(x), y]\) for all \(x, y\in \mathcal{G}.\)

For an element  \(x=x_i \,(i\in \overline{1, m})\) take an element
\(a_i\in \mathcal{G}\) and a number \(\omega_i\in \mathbb{C}\)
such that
\[
\Delta(x_i)=[x_i, a_i]+\omega_i\theta(x_i).
\]
Since \(\Delta(x_i)\in \mathcal{I}\) and \([x_i, a_i]\in
\mathcal{G},\) it follows that  \([x_i, a_i]=0.\) Thus
\[
\Delta(x_i)=\omega_i\theta(x_i)=\omega_i y_i.
\]

Now for the element  \(x=x_i+x_j,\) where \(i\neq j,\) take an
element \(a_{i,j}\in \mathcal{G}\) and a number \(\omega_{i,j}\in
\mathbb{C}\) such that
\[
\Delta(x_i+x_j)=[x_i+x_j, a_{i,j}]+\omega_{i, j}\theta(x_i+x_j)\in
\mathcal{I}.
\]
Then \([x_i+x_j, a_{i,j}]=0.\) Thus
\[
\Delta(x_i+x_j)=\omega_{i,j}\theta(x_i+x_j)=\omega_{i,j}(y_i+y_j).
\]
On the other hand
\[
\Delta(x_i+x_j)=\omega_i x_i+\omega_jy_j.
\]
Comparing the last two equalities we obtain \(\omega_i=\omega_j\) for all \(i, j.\) This means that there
exists a number \(\omega \in \mathbb{C}\) such that
\begin{eqnarray}\label{xxx}
\Delta(x_i)=\omega y_i
\end{eqnarray}
for all \(i=1, \ldots, m.\)

Now for   \(x=x_i+y_i\in \mathcal{G}+\mathcal{I}\) take an element
\(a_x\in \mathcal{G}\) and numbers \(\omega_x, \lambda_x\in
\mathbb{C}\) such that

\[
\Delta(x_i+y_i)=[x_i+y_i, a_x]+\omega_x\theta(x_i)+\lambda_x
y_i\in \mathcal{I}.
\]
Then  \([x_i, a_x]=0,\) and
\[
0=\theta(0)=\theta([x_i, a_x])=[\theta(x_i), a_x]=[y_i, a_x].
\]
Thus
\begin{eqnarray*}
\Delta(x_i+y_i)=\omega_x\theta(x_i)+\lambda_x
y_i=(\omega_x+\lambda_x)y_i.
\end{eqnarray*}
Taking into account  \eqref{xxx} we obtain that
\[
\Delta(y_i)=\Delta(x_i+y_i)-\Delta(x_i)=(\omega_x+\lambda_x)
y_i-\omega y_i=(\omega_x-\omega+\lambda_x)y_i.
\]
This means that for every \(i\in \{1, \ldots, m\}\) there exists a
number \(\lambda_i \in \mathbb{C}\) such that
\begin{eqnarray*}
\Delta(y_i)=\lambda_i y_i
\end{eqnarray*}
for all \(i=1, \ldots, m.\)

Now take an element  \(x=x_i+x_j+y_i+y_j\in
\mathcal{G}+\mathcal{I},\) where \(i\neq j.\) Since
\[
\Delta(x_i+x_j+y_i+y_j)=[x_i+x_j+y_i+y_j,
a_x]+\omega_x\theta(x_i+x_j)+\lambda_x (y_i+y_j)\in \mathcal{I},
\]
we get that  \([x_i+x_j, a_x]=0,\) and therefore \([y_i+y_j,
a_x]=0.\) Thus
\begin{eqnarray*}
\Delta(x_i+x_j+y_i+y_j)=\omega_x\theta(x_i+x_j)+\lambda_x
(y_i+y_j)=(\omega_x+\lambda_x)(y_i+y_j).
\end{eqnarray*}
Taking into account  \eqref{xxx} we obtain that
\begin{eqnarray*}
\Delta(y_i+y_j) & =
&\Delta(x_i+x_j+y_i+y_j)-\Delta(x_i+x_j)=\\
& = & (\omega_x+\lambda_x)(y_i+y_j)-\omega(y_i+y_j)=
(\omega_x-\omega+\lambda_x)(y_i+y_j).
\end{eqnarray*}
On the other hand
\[
\Delta(y_i+y_j)=\Delta(y_i)+\Delta(y_j)=\lambda_i y_i+\lambda_j
y_j.
\]
Comparing the last two equalities we obtain that
\(\lambda_i=\lambda_j\) for all \(i\) and \(j.\) This means that
there exists a number \(\lambda \in \mathbb{C}\) such that
\begin{eqnarray}\label{yyy}
\Delta(y_i)=\lambda y_i
\end{eqnarray}
for all \(i=1, \ldots, m.\) Combining \eqref{xxx} and \eqref{yyy}
we obtain  that \(\Delta=\omega
\theta+\lambda\textrm{pr}_\mathcal{I}.\) This means that
\(\Delta\) is a derivation. The proof is complete.
\end{proof}

Let now \(\Delta\) be an arbitrary local derivation on
\(\mathcal{L}.\) For an  arbitrary element  \(x\in \mathcal{L}\)
take an element \(a_x\in \mathcal{G}\) and a number \(\omega_x\in
\mathbb{C}\) such that
\[
\Delta(x)=[x, a_x]+\omega_x\theta(x).
\]
Then the  mapping
\[
x\in \mathcal{G}\rightarrow [x, a_x]\in \mathcal{G}
\]
is a well-defined local derivation on \(\mathcal{G},\) and
therefore by \cite[Theorem 3.1]{AK2016} it is an inner derivation
generated by an element \(a\in \mathcal{G}.\) Then the local
derivation \(\Delta-R_a\) maps \(\mathcal{L}\) into
\(\mathcal{I}.\) By Lemma~\ref{GeqI} we get that \(\Delta-R_a\) is
a derivation and therefore \(\Delta\) is also a derivation.

In the next Lemma we consider a simple Leibniz algebra
\(\mathcal{L}=\mathcal{G}\dot{+}\mathcal{I}\) such that
\(\mathcal{G}\) and \(\mathcal{I}\) are not isomorphic
\(\mathcal{G}\)-modules.

\begin{lm}\label{GeqII}
Let \(\Delta\) be a local derivation  on \(\mathcal{L}\) such that
\(\Delta\) maps \(\mathcal{L}\) into \(\mathcal{I}.\)  Then
\(\Delta\)  is  a derivation.
\end{lm}

\begin{proof}

Fix a basis  \(\{y_1, \ldots, y_k\}\)  in \(\mathcal{I}.\) We can
assume that for any \(y_i\) there exists a weight~\(\beta_i\) such
that \(y_i\in \mathcal{I}_{\beta_i}.\)

Let \(h_0\) be a strongly regular element in \(\mathcal{H}.\) For
\(y=h_0+y_i\in \mathcal{G}+\mathcal{I}\) take an element \(a_y\in
\mathcal{G}\) and number \(\lambda_y\in \mathbb{C}\) such that
\[
\Delta(h_0+y_i)=[h_0, a_y]+[y_i, a_y]+\lambda_y y_i\in
\mathcal{I}.
\]
Then  \([h_0, a_y]=0,\) and therefore \(a_y\in \mathcal{H}.\)
Further
\[
\Delta(y_i)=\Delta(h_0+y_i)=[y_i, a_y]+\lambda_y
y_i=(\beta_i(a_y)+\lambda_y)y_i.
\]
This means that there exist numbers  \(\lambda_1, \ldots,
\lambda_m\) such that
\[
\Delta(y_i)=\lambda_i y_i
\]
for all \(i=1, \ldots, m.\)

Now we will show that \(\lambda_1=\ldots=\lambda_m.\) Take
\(y_{i_1},\)  \(y_{i_2},\) \(i_1\neq i_2.\)  Denote
\(i_{\beta_1}=y_{i_1},\) \(i_{\beta_2}=y_{i_2}.\) We have
\begin{eqnarray}\label{lambda}
\Delta(i_{\beta_1}+i_{\beta_2})=\lambda_{i_1}i_{\beta_1}+\lambda_{i_2}i_{\beta_2}.
\end{eqnarray}

Without lost of generality we can assume that \(\beta_1\) is a
fixed highest weight  of \(\mathcal{I}.\) It is known \cite[Page
108]{Hum} that difference of two weights represented as
\[
\beta_1-\beta_2=n_1\alpha_1+\ldots+n_l\alpha_l,
\]
where \(\alpha_1, \ldots, \alpha_l\) are simple roots of
\(\mathcal{G},\) \(n_1, \ldots, n_l\) are non negative integers.

Case 1. \(\alpha_0=n_1\alpha_1+\ldots+n_l\alpha_l\) is not a root.
Consider an element
\[
x=n_1e_{\alpha_1}+\ldots+n_le_{\alpha_l}+i_{\beta_1}+i_{\beta_2}.
\]
Take an element \(a_x=h+\sum\limits_{\alpha\in \Phi}c_\alpha
e_\alpha\in \mathcal{G}\) and number \(\lambda_x\) such that
\[
\Delta(x)=[x, a_x]+\lambda_x(i_{\beta_1}+i_{\beta_2}).
\]
Since \(\Delta(x)\in \mathcal{I},\) we obtain that
\[
\left[\sum\limits_{s=1}^ln_s e_{\alpha_s},
h+\sum\limits_{\alpha\in \Phi}c_\alpha e_\alpha\right]=0.
\]
Let us rewrite the last equality as
\[
\sum\limits_{s=1}^l n_s
\alpha_s(h)e_{\alpha_s}+\sum\limits_{t=1}^l\sum\limits_{\alpha\in
\Phi} \ast e_{\alpha+\alpha_t}=0,
\]
where the symbols \(\ast\) denote appropriate coefficients. The
second summand does not contain any element of the form
\(e_{\alpha_s}.\) Indeed, if we assume that
\(\alpha_s=\alpha+\alpha_t,\) we have that
\(\alpha=\alpha_s-\alpha_t.\) But \(\alpha_s-\alpha_t\) is not a
root, because \(\alpha_s, \alpha_t\) are simple roots. Hence all
coefficients of the first summand are zero, i.e.,
\[
n_1 \alpha_1(h)=\ldots=n_l \alpha_l(h)=0.
\]
Further
\[
\Delta(i_{\beta_1}+i_{\beta_2})=\Delta(x)=[i_{\beta_1}+i_{\beta_2},
a_x]+\lambda_x(i_{\beta_1}+i_{\beta_2}).
\]
Let us calculate the commutator \([i_{\beta_1}+i_{\beta_2}, a_x].\)
We have
\begin{eqnarray*}
[i_{\beta_1}+i_{\beta_2}, a_x] & = &
\left[i_{\beta_1}+i_{\beta_2}, h+\sum\limits_{\alpha\in
\Phi}c_\alpha e_\alpha\right]=\\
& = &
\beta_1(h)i_{\beta_1}+\beta_2(h)i_{\beta_2}+\sum\limits_{t=1}^2\sum\limits_{\alpha\in
\Phi}\ast i_{\beta_t+\alpha}.
\end{eqnarray*}
The last summand does not contain \(i_{\beta_1}\) and
\(i_{\beta_2},\) because \(\beta_1-\beta_2\) is not a root by the assumption.  This means that
\begin{eqnarray}\label{iii}
\Delta(i_{\beta_1}+i_{\beta_2})=(\beta_1(h)+\lambda_x)i_{\beta_1}+(\beta_2(h)+\lambda_x)i_{\beta_2}.
\end{eqnarray}
The difference of the coefficients of the right side is
\[
\beta_1(h)-\beta_2(h)=\sum\limits_{s=1}^l n_s\alpha_s(h)=0,
\]
because \(n_1\alpha_1(h)=\ldots=n_l\alpha_l(h)=0.\) Finally,
comparing coefficients in \eqref{lambda} and  \eqref{iii} we get
\[
\lambda_{i_1}=\beta_1(h)+\lambda_x=\beta_2(h)+\lambda_x=\lambda_{i_2}.
\]

Case 2. \(\alpha_0=n_1\alpha_1+\ldots+n_l\alpha_l\) is a root.
Since \(\beta_1\) is a highest weight, we get \(\dim
\mathcal{I}_{\beta_1}=1.\)  Further since  \(\beta_1-\beta_2\) is
a root, \cite[Lemma 3.2.9]{Good} implies that \( \dim
\mathcal{I}_{\beta_2}=\dim \mathcal{I}_{\beta_1},\) and therefore
there exist  numbers \(t_{-\alpha_0}\neq0\) and \(t_{\alpha_0}\)
such that \([i_{\beta_1},
e_{-\alpha_0}]=t_{-\alpha_0}i_{\beta_2},\, [i_{\beta_2},
e_{\alpha_0}]=t_{\alpha_0}i_{\beta_1}.\) Consider an element
\[
x=t_{-\alpha_0}e_{\alpha_0}+t_{\alpha_0}e_{-\alpha_0}+i_{\beta_1}+i_{\beta_2}.
\]
Take an element \(a_x=h+\sum\limits_{\alpha\in \Phi}c_\alpha
e_\alpha\in \mathcal{G}\) and number \(\lambda_x\) such that
\[
\Delta(x)=[x, a_x]+\lambda_x(i_{\beta_1}+i_{\beta_2}).
\]
Since \(\Delta(x)\in \mathcal{I},\) we obtain that
\[
\left[t_{-\alpha_0}e_{\alpha_0}+t_{\alpha_0}e_{-\alpha_0},
h+\sum\limits_{\alpha\in \Phi}c_\alpha e_\alpha\right]=0.
\]
Let us rewrite the last equality as
\[
\alpha_0(h)t_{-\alpha_0}e_{\alpha_0}-\alpha_0(h)t_{\alpha_0}e_{-\alpha_0}+(t_{-\alpha_0}c_{-\alpha_0}-
t_{\alpha_0}c_{\alpha_0})h_{\alpha_0}+\sum\limits_{\alpha\neq\pm\alpha_0}
\ast e_{\alpha\pm\alpha_0}=0,
\]
where \(h_{\alpha_0}=[e_{\alpha_0}, e_{-\alpha_0}]\in
\mathcal{H}.\) The last  summand in the sum does not contain
elements \(e_{\alpha_0}\) and \(e_{-\alpha_0}.\) Indeed, if we
assume that \(\alpha_0=\alpha-\alpha_0,\) we have that
\(\alpha=2\alpha_0.\) But \(2\alpha_0\) is not a root. Hence the
first three coefficients of this sum are zero, i.e.,
\begin{eqnarray}\label{zzz}
\alpha_0(h)=0,\,
t_{\alpha_0}c_{\alpha_0}=t_{-\alpha_0}c_{-\alpha_0}.
\end{eqnarray}
Further
\[
\Delta(i_{\beta_1}+i_{\beta_2})=\Delta(x)=[i_{\beta_1}+i_{\beta_2},
a_x]+\lambda_x(i_{\beta_1}+i_{\beta_2}).
\]
Let us consider the product  \([i_{\beta_1}+i_{\beta_2}, a_x].\)
We have
\begin{eqnarray*}
[i_{\beta_1}+i_{\beta_2}, a_x] & = &
\left[i_{\beta_1}+i_{\beta_2}, h+\sum\limits_{\alpha\in
\Phi}c_\alpha e_\alpha\right]=\\
& = & [i_{\beta_1}, h]+c_{\alpha_0}[i_{\beta_2},
e_{\alpha_0}]+[i_{\beta_2}, h]+c_{-\alpha_0}[i_{\beta_1},
e_{-\alpha_0}]+\\
& + & c_{\alpha_0}[i_{\beta_1},
e_{\alpha_0}]+c_{-\alpha_0}[i_{\beta_2},
e_{-\alpha_0}]+\sum\limits_{t=1}^2\sum\limits_{\alpha\neq\pm\alpha_0}c_\alpha
[i_{\beta_t}, e_\alpha]=\\
& = &
(\beta_1(h)+t_{\alpha_0}c_{\alpha_0})i_{\beta_1}+(\beta_2(h)+t_{-\alpha_0}c_{-\alpha_0})i_{\beta_2}+\\
& + & \ast i_{2\beta_1-\beta_2}+\ast
i_{2\beta_2-\beta_1}+\sum\limits_{t=1}^2
\sum\limits_{\alpha\neq\pm\alpha_0}\ast i_{\beta_t+\alpha}.
\end{eqnarray*}
The last three summands do not contain \(i_{\beta_1}\) and
\(i_{\beta_2},\) because \(\beta_1-\beta_2=\alpha_0\) and
\(\alpha\neq\pm\alpha_0.\) This means that
\begin{eqnarray}\label{ttt}
\Delta(i_{\beta_1}+i_{\beta_2})=(\beta_1(h)+t_{\alpha_0}c_{\alpha_0}+
\lambda_x)i_{\beta_1}+(\beta_2(h)+t_{-\alpha_0}c_{-\alpha_0}+\lambda_x)i_{\beta_2}.
\end{eqnarray}
Taking into account \eqref{zzz} we find the difference of
coefficients in the right side:
\[
(\beta_1(h)+t_{\alpha_0}c_{\alpha_0})-(\beta_2(h)+t_{-\alpha_0}c_{-\alpha_0})=
\alpha_0(h)+t_{\alpha_0}c_{\alpha_0}-t_{-\alpha_0}c_{-\alpha_0}=0.
\]
Combining \eqref{lambda} and \eqref{ttt} we obtain that
\[
\lambda_{i_1}=\beta_1(h)+t_{\alpha_0}c_{\alpha_0}+\lambda_x=\beta_2(h)+t_{-\alpha_0}c_{-\alpha_0}+\lambda_x=\lambda_{i_2}.
\]
So,  we have proved that
\[
\Delta(x_i+y_i)=\lambda y_i,
\]
where \(\lambda\in \mathbb{C}.\) This means that
\(\Delta=\lambda\textrm{pr}_\mathcal{I}.\) The proof is complete.
\end{proof}

Let \(\Delta\) be an arbitrary local derivation and let  \(x\in
\mathcal{G}\) be an arbitrary element. Take an element \(a_x\in
\mathcal{G}\)  such that
\[
\Delta(x)=[x, a_x].
\]
As in the case \(\dim \mathcal{G}=\dim \mathcal{I},\) the mapping
\[
x\in \mathcal{G}\rightarrow [x, a_x]\in \mathcal{G}
\]
is an  inner derivation generated by an element \(a\in
\mathcal{G},\) and  \(\Delta-R_a\) maps \(\mathcal{L}\) into
\(\mathcal{I}.\) This means  that \(\Delta\) is a derivation, that
completes the proof of Theorem~~\ref{thmlocal}.

\subsection{Local derivations on filiform   Leibniz
algebras}\label{filiform}

\

In this subsection we consider a special class of nilpotent
Leibniz algebras, so-called filiform Leibniz algebras, and show
that they admit local derivations which are not derivations.

A nilpotent  Leibniz algebra $\mathcal{L}$ is called \textit{filiform} if
$\dim \mathcal{L}^k=n-k$ for $2\leq k \leq n,$  where
$\mathcal{L}^1=\mathcal{L},$ $\mathcal{L}^{k+1}=[\mathcal{L}^{k},
\mathcal{L}],\, k\geq1.$

\begin{thm}\label{filiform}
Let $\mathcal{L}$ be a finite-dimensional filiform   Leibniz
algebra with $\dim \mathcal{L}\geq 3.$ Then $\mathcal{L}$ admits a
local derivation which is not a derivation.
\end{thm}

For  filiform Lie algebras this result was proved in \cite[Theorem
4.1]{AK2016}. Hence  it is suffices to consider filiform non-Lie
Leibniz algebras.

Recall \cite{OmRa} that each complex $n$-dimensional filiform Leibniz algebra admits a basis $\{e_1,e_2,\dots,e_n\}$
such that the table of multiplication of the algebra has one of the following forms:
$$\begin{array}{rl}
F_1(\alpha_4,\alpha_5,\dots,\alpha_n,\theta):&\left\{\begin{array}{ll}
[e_1,e_1]=e_3,&\\{} [e_i,e_1]=e_{i+1},&2\leq i\leq n-1,\\{}
[e_1,e_2]=\displaystyle\sum_{k=4}^{n-1} \alpha_k e_k+\theta
e_n,&\\{} [e_i,e_2]=\displaystyle\sum_{k=i+2}^n \alpha_{k+2-i}
e_k,&2\leq i\leq n-2;
\end{array}\right.
\end{array}
$$

$$
\begin{array}{rl}
F_2(\beta_3,\beta_4,\dots,\beta_n,\gamma):&\left\{\begin{array}{ll}
[e_1,e_1]=e_3,&\\{} [e_i,e_1]=e_{i+1},& 3\leq i\leq n-1,\\{}
[e_1,e_2]=\displaystyle\sum_{k=4}^{n} \beta_k e_k,&\\{}
[e_2,e_2]=\gamma e_n,&\\{} [e_i,e_2]=\displaystyle\sum_{k=i+2}^n
\beta_{k+2-i} e_k,& 3\leq i\leq n-2;
\end{array}\right.
\end{array}
$$
$$
\begin{array}{rl}
F_3(\theta_1,\theta_2,\theta_3):&\left\{\begin{array}{ll}
[e_i,e_1]=e_{i+1},& 2\leq i\leq n-1,\\{} [e_1,e_i]=-e_{i+1},&
3\leq i\leq n-1,\\{} [e_1,e_1]=\theta_1 e_n,&\\{}
[e_1,e_2]=-e_3+\theta_2 e_n,&\\{} [e_2,e_2]=\theta_3 e_n,&\\{}
[e_i,e_j]=-[e_j,e_i]\in \textrm{span} \{e_{i+j+1},\ldots,
e_n\},&\left\{\begin{array}{l} 1\leq i\leq n-2,\\{}
 2\leq j\leq n-i,\ i<j
 \end{array}\right.\\[2mm]
[e_i,e_{n-i}]=-[e_{n-i},e_i]=\alpha (-1)^{i} e_n ,&2\leq i\leq
n-1,
\end{array}\right.
\end{array}$$
where $\alpha\in \{0,1\}$ for odd $n$ and $\alpha=0$ for even $n$.
Moreover, the structure constants of an algebra from
$F_3(\theta_1,\theta_2,\theta_3)$ should satisfy the Leibniz
identity.

It is easy to see that algebras of the first and the second
families are non-Lie algebras. Moreover, an algebra of the third
family is a Lie algebra if and only if
$(\theta_1,\theta_2,\theta_3)=(0,0,0).$

Firstly we consider the  family of algebras
$\mathcal{L}=F_1(\alpha_4,\alpha_5,\dots,\alpha_n,\theta).$ Let us
define a linear operator $D$ on $\mathcal{L}$ by
\begin{equation}\label{filiformder}
D\left(\sum\limits_{k=1}^n x_ke_k\right)=\alpha (x_1 + x_2)e_{n-1}
+ \beta x_3 e_n,
\end{equation}
where $\alpha, \beta\in \mathbb{C}.$

\begin{lm}\label{filider}
The linear operator $D$ on $\mathcal{L}$ defined
by~\eqref{filiformder} is a derivation if and only
if~$\alpha=\beta.$
\end{lm}

\begin{proof}   Suppose that  the linear operator  $D$ defined by~\eqref{filiformder} is a  derivation.
Since $[e_1, e_1]=e_3,$ we have that
\begin{eqnarray*}
D([e_1, e_1]) & = & D(e_3)=\beta e_n
\end{eqnarray*}
and
\begin{eqnarray*}
[D(e_1), e_1]+ [e_1, D(e_1)] & = & [\alpha e_{n-1}, e_1]+[e_1,
\alpha e_{n-1}]=\alpha e_n.
\end{eqnarray*}
Thus $\alpha=\beta.$

Conversely, let $D$ be a linear operator defined by
\eqref{filiformder} with $\alpha=\beta.$ We may assume that
$\alpha=\beta=1.$

In order to prove that $D$ is a derivation it is sufficient to
show that
$$
D([e_i, e_j])=[D(e_i),e_j]+[e_i,D(e_j)]
$$
for all $1\leq i, j \leq n.$

Case 1. $i+j=2.$ Then $i=j=1$ and in this case we can check as
above.

Case 2. $i+j=3.$ Then
\[
[D(e_2),e_1]+[e_2,D(e_1)]=[e_{n-1}, e_1]+[e_2, e_{n-1}]=[e_{n-1},
e_1]=e_n=D(e_3)=D([e_2, e_1])
\]
and
\[
[D(e_1),e_2]+[e_1,D(e_2)]=[e_{n-1}, e_2]+[e_1,
e_{n-1}]=0=D\left(\displaystyle\sum_{k=4}^{n-1} \alpha_k
e_k+\theta e_n\right)=D([e_1, e_2]).
\]

Case 3. $i+j\geq 4.$ Then
\begin{eqnarray*}
D([e_i,  e_j])=0=[D(e_i), e_j]+[e_i, D(e_j)].
\end{eqnarray*}
The proof is complete.
\end{proof}

Now we consider  the linear operator $\Delta$ defined by
\eqref{filiformder} with $\alpha=1, \beta=0.$

\begin{lm}\label{locfili}
The  linear operator $\Delta$ is a local derivation which is not a
derivation.
\end{lm}

\begin{proof} By Lemma~\ref{filider}, $\Delta$ is not a
derivation.

Let us show that $\Delta$ is a local derivation. Denote by $D_1$
the derivation defined by~\eqref{filiformder} with
$\alpha=\beta=1.$ Let $D_2$ be a linear operator on $\mathcal{L}$
defined by
\begin{equation*}
D_2\left(\sum\limits_{k=1}^n x_ke_k\right)=(x_1+x_2) e_n.
\end{equation*}
Since $D_2|_{\mathcal{[\mathcal{L},\mathcal{L}]}}\equiv 0$ and
$D_2(\mathcal{L})\subseteq Z(\mathcal{L}),$ it follows that
$$
D_2([x,y])=0=[D_2(x),y]+[x,D_2(y)]
$$
for all $x, y\in \mathcal{L}.$ So, $D_2$ is a derivation.

Finally, for any $x=\sum\limits_{k=1}^n x_ke_k$ we find a
derivation $D$ such that $\Delta(x)=D(x).$

Case 1. $x_1+x_2=0.$ Then
\begin{eqnarray*}
\Delta(x) =  0 =  D_2(x).
\end{eqnarray*}

Case 2. $x_1+x_2\neq 0.$ Set
$$
D=D_1+tD_2,
$$
where $t=-\frac{\textstyle x_3}{\textstyle x_1+x_2}.$ Then
\begin{eqnarray*}
D(x)  & = &  D_1(x)+t D_2(x) =
(x_1+x_2)e_{n-1}+x_3e_n+t(x_1+x_2)e_n =
\\
& = & (x_1+x_2)e_{n-1}+x_3e_n-x_3e_n =\Delta(x).
\end{eqnarray*}
The proof is complete.
\end{proof}

Now let us consider the second and third classes.

For an algebra
$\mathcal{L}=F_2(\beta_3,\beta_4,\dots,\beta_n,\gamma)$ from the
second class define a linear operator $D$ on $\mathcal{L}$ by
\begin{equation}\label{ftwo}
D\left(\sum\limits_{k=1}^n x_ke_k\right)=\alpha x_1 e_{n-1} +
\beta x_3 e_n,
\end{equation}
where $\alpha, \beta\in \mathbb{C}.$

For an algebra $\mathcal{L}=F_3(\theta_1,\theta_2,\theta_3)$
define a linear operator $D$ on $\mathcal{L}$ by
\begin{equation}\label{fth}
D\left(\sum\limits_{k=1}^n x_ke_k\right)=\alpha x_2 e_{n-1}+\beta
x_3 e_{n},
\end{equation}
where $\alpha, \beta\in \mathbb{C}.$

As in the proof of Lemma~\ref{filider},  we can check that the
operator \(D\) defined by \eqref{ftwo} or \eqref{fth} is a
derivation if and only if \(\alpha=\beta.\) Therefore the operator \(D\)
defined by \eqref{ftwo} or  \eqref{fth} gives is the example of a
local derivations which is  not a derivation.

\section{Automorphisms of simple Leibniz algebras}
\label{fifth}

Let $\mathcal{L}=\mathcal{G}\dot{+} \mathcal{I}$ be a simple Leibniz algebra, where
$\mathcal{G}$ is a simple Lie algebra and $\mathcal{I}$ is its right irreducible module.

We consider an  automorphism $\varphi$ of $\mathcal{L}$. The ideal
generated by squares of elements of $\mathcal{L}$ coincides with
the \(\textrm{span}\{ [x, x]: x\in \mathcal{L}\}.\) Then for
$x=\sum\limits_{k=1}^n\lambda_i[x_i, x_i]\in \mathcal{I}$ we have
\begin{eqnarray*}
\varphi(x) & = & \varphi\left(\sum\limits_{k=1}^n\lambda_i[x_i,x_i]\right)=\sum\limits_{k=1}^n\lambda_i\varphi([x_i,x_i])=\\
&=& \sum\limits_{k=1}^n\lambda_i[\varphi(x_i),\varphi(x_i)]=
\sum\limits_{k=1}^n\lambda_i[y_i,y_i],
\end{eqnarray*}
i.e., \(\varphi(\mathcal{I})\subseteq \mathcal{I}.\)

Now let $y=\sum\limits_{k=1}^n\lambda_i[y_i,y_i]$  be  an
arbitrary element of the ideal $\mathcal{I}.$ Since $\varphi$ is an
automorphism,  for every $y_i\in \mathcal{L}$ there exists  $x_i\in
\mathcal{L}$ such that $\varphi(x_i)=y_i$. Then we have
\begin{eqnarray*}
y&=&\sum\limits_{k=1}^n\lambda_i[y_i,y_i]=\sum\limits_{k=1}^n\lambda_i[\varphi(x_i),\varphi(x_i)]=\\
&=&\sum\limits_{k=1}^n \lambda_i\varphi([x_i,x_i])=
\varphi\left(\sum\limits_{k=1}^n\lambda_i[x_i,x_i]\right).
\end{eqnarray*}
This implies that for the element
$z=\sum\limits_{k=1}^n\lambda_i[x_i,x_i]\in \mathcal{I}$ we have
$\varphi(z)=y$. So we have proved that $\mathcal{I} \subseteq
\varphi(\mathcal{I}),$ and therefore
$\varphi(\mathcal{I})=\mathcal{I}$.

Now we shall show  that  any $\varphi\in
\textrm{Aut}(\mathcal{L})$ can be represented as the sum
$\varphi=\varphi_{\mathcal{G},\mathcal{G}}+\varphi_{\mathcal{G},\mathcal{I}}+\varphi_{\mathcal{I},\mathcal{I}}$,
where \(\varphi_{\mathcal{G},\mathcal{G}} : \mathcal{G}
\rightarrow \mathcal{G}\) is an automorphism on \(\mathcal{G},\)
\(\varphi_{\mathcal{G},\mathcal{I}} : \mathcal{G} \rightarrow
\mathcal{I}\) is a \(\mathcal{G}\)-module homomorphism from
\(\mathcal{G}\) into \(\mathcal{I},\) and
\(\varphi_{\mathcal{I},\mathcal{I}} :\mathcal{I} \rightarrow
\mathcal{I}\) is a \(\mathcal{G}\)-module isomorphism  of
\(\mathcal{I}.\) In particular,
\[
\varphi(x+i)=\varphi_{\mathcal{G},\mathcal{G}}(x)+\varphi_{\mathcal{G},\mathcal{I}}(x)+\varphi_{\mathcal{I},\mathcal{I}}(i),\,
x+i\in \mathcal{G}+\mathcal{I}.
\]

\begin{lm}\label{decoat}
Let $\mathcal{L}=\mathcal{G}\dot{+} \mathcal{I}$ be a simple
complex Leibniz algebra and let \(\varphi\in
\textrm{Aut}(\mathcal{L})\) be an automorphism. Then
\[
\varphi=\varphi_{\mathcal{G},\mathcal{G}}+\varphi_{\mathcal{I},\mathcal{I}},
\]
if \(\dim \mathcal{G}\neq\dim \mathcal{I},\) and
\[
\varphi=\varphi_{\mathcal{G},\mathcal{G}}+\omega \theta
\circ\varphi_{\mathcal{G},\mathcal{G}}+\varphi_{\mathcal{I},\mathcal{I}},
\]
if \(\dim \mathcal{G}=\dim \mathcal{I}.\)
\end{lm}

\begin{proof}
Let $x,y\in \mathcal{G}$, then
\begin{eqnarray*}
\varphi_{\mathcal{G},\mathcal{G}}([x,y])+\varphi_{\mathcal{G},\mathcal{I}}([x,y])&=&
\varphi([x,y])=[\varphi(x),\varphi(y)]=\\
&=&
[\varphi_{\mathcal{G},\mathcal{G}}(x)+\varphi_{\mathcal{G},\mathcal{I}}(x),
\varphi_{\mathcal{G},\mathcal{G}}(y)+\varphi_{\mathcal{G},\mathcal{I}}(y)]=\\
&=&[\varphi_{\mathcal{G},\mathcal{G}}(x),\varphi_{\mathcal{G},\mathcal{G}}(y)]+[\varphi_{\mathcal{G},\mathcal{I}}(x),
\varphi_{\mathcal{G},\mathcal{G}}(y)].
\end{eqnarray*}
This implies
\begin{eqnarray}\label{ggggg}
\varphi_{\mathcal{G},\mathcal{G}}([x,y])=[\varphi_{\mathcal{G},\mathcal{G}}(x),\varphi_{\mathcal{G},\mathcal{G}}(y)],
\end{eqnarray}
\begin{eqnarray}\label{gigigi}
\varphi_{\mathcal{G},\mathcal{I}}([x,y])=[\varphi_{\mathcal{G},\mathcal{I}}(x),
\varphi_{\mathcal{G},\mathcal{G}}(y)].
\end{eqnarray}

Set
\[
\psi=\varphi_{\mathcal{G},\mathcal{G}}+\varphi_{\mathcal{I},\mathcal{I}}.
\]
Let us show that \(\psi\) is also an automorphism. Indeed,
\begin{eqnarray*}
\psi([x+i,
y+j])&=&\psi([x,y]+[i,y])=\varphi_{\mathcal{G},\mathcal{G}}([x,y])+\varphi_{\mathcal{I},\mathcal{I}}([i,y])=\\
&=&
[\varphi_{\mathcal{G},\mathcal{G}}(x),\varphi_{\mathcal{G},\mathcal{G}}(y)]+\varphi([i,y])=[\varphi_{\mathcal{G},\mathcal{G}}(x),
\varphi_{\mathcal{G},\mathcal{G}}(y)]+[\varphi(i),
\varphi(y)]=\\
&=&
[\varphi_{\mathcal{G},\mathcal{G}}(x),\varphi_{\mathcal{G},\mathcal{G}}(y)]+[\varphi_{\mathcal{I},\mathcal{I}}(i),
\varphi_{\mathcal{G},\mathcal{G}}(y)+\varphi_{\mathcal{G},\mathcal{I}}(y)]=\\
&=&
[\varphi_{\mathcal{G},\mathcal{G}}(x),\varphi_{\mathcal{G},\mathcal{G}}(y)]+[\varphi_{\mathcal{I},\mathcal{I}}(i),
\varphi_{\mathcal{G},\mathcal{G}}(y)]=\\&=&
[\varphi_{\mathcal{G},\mathcal{G}}(x)+\varphi_{\mathcal{I},\mathcal{I}}(i),\varphi_{\mathcal{G},\mathcal{G}}(y)+
\varphi_{\mathcal{I},\mathcal{I}}(j)]=[\psi(x+i),\psi(y+j)].
\end{eqnarray*}

 Now consider an automorphism
\[
\eta=\varphi\circ\psi^{-1}.
\] Then
\[
\eta(x+i)=x+\eta_{\mathcal{G},\mathcal{I}}(x)+i,
\]
where
\(\eta_{\mathcal{G},\mathcal{I}}=\varphi_{\mathcal{G},\mathcal{I}}\circ\varphi^{-1}_{\mathcal{G},\mathcal{G}}.\)
Applying \eqref{ggggg} and \eqref{gigigi} to \(\eta\) we obtain that
\[
\eta_{\mathcal{G},\mathcal{I}}([x,y])=[\eta_{\mathcal{G},\mathcal{I}}(x),
y].
\]
This means that \(\eta_{\mathcal{G},\mathcal{I}}\) is a
\(\mathcal{G}\)-module homomorphism from \(\mathcal{G}\) into
\(\mathcal{I}.\)

Case 1.  Let \(\dim \mathcal{G}\neq \dim \mathcal{I}.\) In this
case   by Schur's Lemma we obtain that
\(\eta_{\mathcal{G},\mathcal{I}}=0.\) Now the equality
\(\eta_{\mathcal{G},\mathcal{I}}=\varphi_{\mathcal{G},\mathcal{I}}\circ\varphi^{-1}_{\mathcal{G},\mathcal{G}}\)
implies that \(\varphi_{\mathcal{G},\mathcal{I}}=0.\) Thus
\[
\varphi=\varphi_{\mathcal{G},\mathcal{G}}+\varphi_{\mathcal{I},\mathcal{I}}.
\]

Case 2. Let \(\dim \mathcal{G}= \dim \mathcal{I}.\) In this case
again  by Schur's Lemma we obtain that
\(\eta_{\mathcal{G},\mathcal{I}}=\omega \theta,\) where
\(\omega\in \mathbb{C}.\) Thus
\(\varphi_{\mathcal{G},\mathcal{I}}=\eta_{\mathcal{G},\mathcal{I}}
\circ\varphi_{\mathcal{G},\mathcal{G}}=\omega \theta
\circ\varphi_{\mathcal{G},\mathcal{G}},\) and therefore
\[
\varphi=\varphi_{\mathcal{G},\mathcal{G}}+\omega \theta
\circ\varphi_{\mathcal{G},\mathcal{G}}+\varphi_{\mathcal{I},\mathcal{I}}.
\]
The proof is complete. \end{proof}

Further we shall use the following lemma.

\begin{lm}\label{iden}
Let \(\varphi\in \textrm{Aut}(\mathcal{L})\) be an automorphism.
Then \begin{itemize}
\item[a)]  if
\(\varphi_{\mathcal{G},\mathcal{G}}=id_\mathcal{G},\) then
\(\varphi_{\mathcal{I},\mathcal{I}}=\lambda id_\mathcal{I};\)
\item[b)]  if
\(\varphi_{\mathcal{I},\mathcal{I}}=id_\mathcal{I},\) then
\(\varphi_{\mathcal{G},\mathcal{G}}= id_\mathcal{G}.\)
\end{itemize}
\end{lm}

\begin{proof} a) Similar to the proof of \eqref{gigigi} we obtain that
\begin{eqnarray*}
\varphi_{\mathcal{I},\mathcal{I}}([i,x])=[\varphi_{\mathcal{I},\mathcal{I}}(i),
\varphi_{\mathcal{G},\mathcal{G}}(x)]=[\varphi_{\mathcal{I},\mathcal{I}}(i),
x]
\end{eqnarray*}
for all \(i\in \mathcal{I}, x\in \mathcal{G}.\) By Schur's Lemma
we obtain that \(\varphi_{\mathcal{I},\mathcal{I}}=\lambda
id_\mathcal{I}.\)

b) Since \begin{eqnarray*} [i,
x]=\varphi_{\mathcal{I},\mathcal{I}}([i,x])=[\varphi_{\mathcal{I},\mathcal{I}}(i),
\varphi_{\mathcal{G},\mathcal{G}}(x)]=[i,
\varphi_{\mathcal{G},\mathcal{G}}(x)],
\end{eqnarray*}
we obtain that \([i, \varphi_{\mathcal{G},\mathcal{G}}(x)-x]=0\)
for all \(i\in \mathcal{I}, x\in \mathcal{G}.\) As in the proof of
Lemma~\ref{fourth} we have \([\mathcal{I},
\mathcal{G}_{\varphi_{\mathcal{G},\mathcal{G}}(x)-x}]=0,\) and
\(\varphi_{\mathcal{G},\mathcal{G}}(x)-x=0.\)  The proof is
complete.
\end{proof}

\begin{lm}\label{form}
Let $\mathcal{L}=\mathcal{G}\dot{+} \mathcal{I}$ be a simple
complex Leibniz algebra with \(\dim \mathcal{G}=\dim
\mathcal{I}.\) Then any automorphism  \(\varphi\in
\textrm{Aut}(\mathcal{L})\) can be  represented as
\begin{eqnarray*}
\varphi=\varphi_{\mathcal{G},\mathcal{G}}+\omega \theta
\circ\varphi_{\mathcal{G},\mathcal{G}}+\lambda
\theta\circ\varphi_{\mathcal{G},\mathcal{G}}\circ\theta^{-1},
\end{eqnarray*}
where \(\omega\in \mathbb{C}\) and \(0\neq\lambda\in \mathbb{C}.\)
\end{lm}

\begin{proof}
Let \(\varphi_{\mathcal{G},\mathcal{G}}\) be an automorphism of
\(\mathcal{G}.\) Let us show that
\(\phi=\varphi_{\mathcal{G},\mathcal{G}}+\lambda
\theta\circ\varphi_{\mathcal{G},\mathcal{G}}\circ\theta^{-1}\) is
an automorphism. It suffices to check that
\[
\phi([i, x])=[\phi(i), \phi(x)]
\]
for all \(x\in \mathcal{G},\) \(i\in \mathcal{I}.\) Since
\(\theta\) is a \(\mathcal{G}\)-module isomorphism, it follows
that \(\theta([\theta^{-1}(i), x])=[i,x].\) We have
\begin{eqnarray*}
\phi([i,x]) & = & \lambda
\theta\circ\varphi_{\mathcal{G},\mathcal{G}}\circ\theta^{-1}([i,x])=\lambda
\theta\circ\varphi_{\mathcal{G},\mathcal{G}}\circ\theta^{-1}(\theta([\theta^{-1}(i),
x])=\\
&=& \lambda
\theta\circ\varphi_{\mathcal{G},\mathcal{G}}([\theta^{-1}(i),x])=\lambda
\theta([\varphi_{\mathcal{G},\mathcal{G}}(\theta^{-1}(i)),\varphi_{\mathcal{G},\mathcal{G}}(x)])=\\
&=&
[(\lambda\theta\circ\varphi_{\mathcal{G},\mathcal{G}}\circ\theta^{-1})(i),\varphi_{\mathcal{G},\mathcal{G}}(x)]=
[\phi(i),\phi(x)].
\end{eqnarray*}

Let us consider
\begin{eqnarray*}
\varphi=\varphi_{\mathcal{G},\mathcal{G}}+\omega
\theta\circ\varphi_{\mathcal{G},\mathcal{G}}+\varphi_{\mathcal{I},\mathcal{I}}.
\end{eqnarray*}
Set
\begin{center}
\(\psi=\varphi_{\mathcal{G},\mathcal{G}}+
\varphi_{\mathcal{I},\mathcal{I}},\,\)
\(\phi=\varphi_{\mathcal{G},\mathcal{G}}+
\theta\circ\varphi_{\mathcal{G},\mathcal{G}}\circ\theta^{-1}\) and
\(\eta=\psi\circ\phi^{-1}.\)
\end{center}
Then
\(\eta=\textrm{id}_{\mathcal{G}}+\eta_{\mathcal{I},\mathcal{I}},\)
and therefore by Lemma~\ref{iden}, it follows that
\(\eta_{\mathcal{I},\mathcal{I}}=\lambda
\textrm{id}_\mathcal{I}.\) Thus
\[
\psi=\eta\circ\phi=(\textrm{id}_{\mathcal{G}}+\lambda
\textrm{id}_\mathcal{I})\circ(\varphi_{\mathcal{G},\mathcal{G}}+
\theta\circ\varphi_{\mathcal{G},\mathcal{G}}\circ\theta^{-1})=\varphi_{\mathcal{G},\mathcal{G}}+\lambda
\theta\circ\varphi_{\mathcal{G},\mathcal{G}}\circ\theta^{-1}.
\]
Hence
\[
\varphi=\varphi_{\mathcal{G},\mathcal{G}}+\omega \theta
\circ\varphi_{\mathcal{G},\mathcal{G}}+\lambda
\theta\circ\varphi_{\mathcal{G},\mathcal{G}}\circ\theta^{-1}.
\]
The proof is complete.
\end{proof}

\section{Local  and 2-Local automorphisms on simple Leibniz algebras}
\label{six}

\subsection{2-Local automorphisms of simple Leibniz algebras}

\

Let  $\mathcal{L}=\mathcal{G}\dot{+} \mathcal{I}$ be  a complex
simple Leibniz algebra. Then any $\varphi\in Aut(\mathcal{L})$
decomposes into
\begin{eqnarray*}
\varphi=\varphi_{\mathcal{G},\mathcal{G}}+\varphi_{\mathcal{G},\mathcal{I}}+\varphi_{\mathcal{I},\mathcal{I}},
\end{eqnarray*}
\begin{eqnarray*}
\varphi_{\mathcal{G},\mathcal{I}}=\omega\theta\circ\varphi_{\mathcal{G},\mathcal{G}}
\end{eqnarray*}
where \(\omega\in \mathbb{C}\) and
$\varphi_{\mathcal{G},\mathcal{I}}$ is assumed to be zero when
$\dim \mathcal{G}\neq \dim \mathcal{I}.$

\begin{lm}\label{lm01}
Let $\mathcal{L}=\mathcal{G}\dot{+}\mathcal{I}$ be a simple
Leibniz algebra and let $\varphi\in Aut(\mathcal{L})$ be such that
$\varphi(h_0)=h_0,$ where \(h_0\) is a strongly regular element
from \(\mathcal{H}.\) Then
\begin{itemize}
\item[a)] \(\varphi(e_\alpha)=t_\alpha e_\alpha\)
and  \(\varphi(e_{-\alpha})=t_\alpha^{-1} e_{-\alpha},\) where
\(t_\alpha\in \mathbb{C}\) for all \(\alpha\in \Phi,\)
\item[b)]  \(\varphi(h)=h\) for all \(h\in  \mathcal{H}.\)
\end{itemize}
\end{lm}

\begin{proof}
Let \(
\varphi=\varphi_{\mathcal{G},\mathcal{G}}+\varphi_{\mathcal{G},\mathcal{I}}+\varphi_{\mathcal{I},\mathcal{I}}.
\) Since
\begin{eqnarray*}
h_0=\varphi(h_0)= \varphi_{\mathcal{G},\mathcal{G}}(h_0)+
\varphi_{\mathcal{G},\mathcal{I}}(h_0),
\end{eqnarray*}
it follows that \(\varphi_{\mathcal{G},\mathcal{G}}(h_0)=h_0\) and
\(\varphi_{\mathcal{G},\mathcal{I}}(h_0)=0\) (that is $\theta(h_0)=0$). Thus
\(\varphi_{\mathcal{G},\mathcal{I}}\equiv 0.\) Now assertions a)
and b) follows from \cite[Lemma 2.2]{AK2016A}. The proof is
complete.
\end{proof}

Let \(\mathcal{H}\) be a Cartan subalgebra of  $\mathcal{G}$ and
let
$$
\mathcal{G}=\mathcal{H}\oplus\bigoplus_{\alpha \in
\Phi}\mathcal{G}_{\alpha}, \quad \mathcal{I}=\bigoplus_{\beta\in \Gamma}\mathcal{I}_\beta.
$$
It is known \cite[P. 108]{Hum} that the $\mathcal{G}$-module
$\mathcal{I}$ is generated by the elements of the form
\begin{eqnarray}\label{for03}
[\dots[[i^{+}, e_{-\alpha_1}], e_{-\alpha_2}], \dots,
e_{-\alpha_k}],
\end{eqnarray}
where $i^{+}$ is a highest weight vector of $\mathcal{G}$-module
$\mathcal{I},$ $e_{-\alpha_i}\in \mathcal{G}_{-\alpha_i}$ and
$\alpha_i \in \Phi$ is a  positive   root for all \(i=1,\ldots,
k,\) \(k\in \mathbb{N}.\)

Let $\dim \mathcal{I}_\beta=s_\beta, \ \beta\in \Gamma$ and let
$\left\{i^{(1)}_\beta, \ldots, i^{(s_\beta)}_\beta\right\}$ be a
basis of $\mathcal{I}_\beta,$ consisting of  elements of the form
\eqref{for03}. Set
\begin{eqnarray}\label{inol}
i_0=\sum_{\beta\in \Gamma}\sum_{k=1}^{s_\beta}i^{(k)}_{l_\beta}.
\end{eqnarray}

\begin{lm} \label{lm02} Let $\mathcal{L}=\mathcal{G}\dot{+}
\mathcal{I}$ be a simple Leibniz algebra with  \(\dim
\mathcal{G}\neq\dim \mathcal{I}\) and let $\varphi\in
Aut(\mathcal{L})$ be such that $\varphi(h_0+i_0)=h_0+i_0.$ Then
$\varphi$ is an identity automorphism of $\mathcal{L}.$
\end{lm}

\begin{proof}  Since
\[
h_0+i_0=\varphi(h_0+i_0)=\varphi_{\mathcal{G},\mathcal{G}}(h_0)+\varphi_{\mathcal{I},\mathcal{I}}(i_0),
\]
it follows that $\varphi_{\mathcal{G},\mathcal{G}}(h_0)=h_0$ and
$\varphi_{\mathcal{I},\mathcal{I}}(i_0)=i_0.$ By Lemma~\ref{lm01}
\(\varphi_{\mathcal{G},\mathcal{G}}\) acts diagonally
on~\(\mathcal{G},\) i.e.,
\(\varphi_{\mathcal{G},\mathcal{G}}(e_\alpha)=t_\alpha e_\alpha,\)
\(t_\alpha\in \mathbb{C},\) \(\alpha\in \Phi.\)

Let \(i_+\) be the highest weight vector of \(\mathcal{I}\) and let
\(\alpha\) be a positive root. Then
\begin{eqnarray*}
&[\varphi(i_+), e_\alpha] & =  t_\alpha^{-1}[\varphi(i_+),
\varphi(e_\alpha)]=t_\alpha^{-1}\varphi([i_+,
e_\alpha])=\varphi(0)=0 \\
& [\varphi(i_+), h] & =  [\varphi(i_+),
\varphi(h)]=\alpha_+(h)\varphi(i_+),
\end{eqnarray*}
where \(\alpha_+\in \Gamma\) is a highest weight. This means that
\(\varphi(i_+)\) is also a highest weight vector. Since the
highest weight subspace is one-dimensional, it follows that
$\varphi(i_{+})=\lambda_+ i_{+},$ where \(\lambda_+\in
\mathbb{C}.\) Now taking into account that
$\varphi(i_{+})=\lambda_+ i_{+}$ from previous lemma we conclude
that
\begin{eqnarray*}
\varphi\left(i^{(k)}_{\beta}\right)&=&\lambda_+[\dots[[i_{+},
\varphi(e_{-\alpha_1})], \varphi(e_{-\alpha_2})],\dots,
\varphi(e_{-\alpha_k})]=\\
&=& \left(\prod\limits_{p=1}^k
t_{-\alpha_p}\lambda_+\right)[\dots[[i_{+}, e_{-\alpha_1}],
e_{-\alpha_2}],\dots,
e_{-\alpha_k}]=c_\beta^{(k)}\lambda_+i_\beta^{(k)},
\end{eqnarray*}
i.e., \begin{eqnarray*}
\varphi\left(i^{(k)}_{\beta}\right)=c_\beta^{(k)}\lambda_+i_\beta^{(k)},
\end{eqnarray*}
for \( 1\leq k\leq s_\beta, \beta\in \Gamma.\)  Taking into
account these equalities, from
\(i_0=\varphi_{\mathcal{I},\mathcal{I}}(i_0),\) we obtain
$\varphi(i^{(k)}_{\beta})=i^{(k)}_{\beta}$ for all for \( 1\leq
k\leq s_\beta, \beta\in \Gamma.\) This imply $\varphi(i)=i$ for
all $i\in \mathcal{I}.$ By Lemma~\ref{iden}, it follows that
\(\varphi=id_\mathcal{L}.\) The proof is complete.
\end{proof}

Let \(\dim \mathcal{H}=k\geq 2\) and let \(\alpha_1, \ldots,
\alpha_k\) be simple roots. Set \(i_s=\theta(h_{\alpha_s}),\)
\(s=1,\ldots, k,\) \(i_\alpha=\theta(e_\alpha),\) \(\alpha\in
\Phi.\) Take
\[
i_0=\sum\limits_{s=1}^k i_s+\sum\limits_{\alpha\in\Phi} i_\alpha.
\]

\begin{lm}\label{lm57}
Let $\mathcal{L}=\mathcal{G}\dot{+}\mathcal{I}$ be a complex
simple Leibniz algebra with \(\dim\mathcal{G}=\dim \mathcal{I}\)
and \(\dim \mathcal{H}=k\geq 2.\) Suppose that \(\varphi\) is  an
automorphism on $\mathcal{L}$ such that
\(\varphi(h_0+i_0)=h_0+i_0.\) Then
\(\varphi=\textrm{id}_\mathcal{L}.\)
\end{lm}

\begin{proof} Let
\(
\varphi=\varphi_{\mathcal{G},\mathcal{G}}+\varphi_{\mathcal{G},\mathcal{I}}+\varphi_{\mathcal{I},\mathcal{I}}.
\) Since
\begin{eqnarray*}
\varphi_{\mathcal{G},\mathcal{G}}(h_0)=h_0,\,
\varphi_{\mathcal{G},\mathcal{I}}(h_0)+\varphi_{\mathcal{I},\mathcal{I}}(i_0)=i_0,
\end{eqnarray*}
by Lemma~\ref{lm01} we have that
\(\varphi_{\mathcal{G},\mathcal{G}}(e_\alpha)=t_\alpha e_\alpha,\)
\(\alpha\in \Phi\) and \(\varphi_{\mathcal{G},\mathcal{G}}(h)=h\)
for all \(h\in \mathcal{H}.\) Then
\begin{eqnarray*}
\varphi_{\mathcal{I},\mathcal{I}}(i_s)=\lambda
\theta(\varphi_{\mathcal{G},\mathcal{G}}(\theta^{-1}(i_s)))=\lambda
\theta(\varphi_{\mathcal{G},\mathcal{G}}(h_{\alpha_s}))=\lambda
\theta(h_{\alpha_s})=\lambda i_s
\end{eqnarray*}
and
\begin{eqnarray*}
\varphi_{\mathcal{I},\mathcal{I}}(i_\alpha)=\lambda
\theta(\varphi_{\mathcal{G},\mathcal{G}}(\theta^{-1}(i_\alpha)))=\lambda
\theta(\varphi_{\mathcal{G},\mathcal{G}}(e_{\alpha}))=\lambda
\theta(e_{\alpha})=\lambda t_\alpha i_\alpha.
\end{eqnarray*}
Further
\begin{eqnarray*}
i_0&=&\varphi_{\mathcal{G},\mathcal{I}}(h_0)+\varphi_{\mathcal{I},\mathcal{I}}(i_0)=
\omega\theta(h_0)+\varphi_{\mathcal{I},\mathcal{I}}(i_0)=\\
&=&\omega \sum\limits_{s=1}^k a_s i_s+\sum\limits_{s=1}^k \lambda
i_s+\sum\limits_{\alpha\in\Phi} \lambda t_\alpha i_\alpha.
\end{eqnarray*}
i.e.,
\begin{eqnarray*}
\sum\limits_{\alpha\in\Phi} (1-\lambda t_\alpha) i_\alpha=
\sum\limits_{s=1}^k (\omega a_s +\lambda -1)i_s.
\end{eqnarray*}
Since the right side of this equality belongs to \(\mathcal{I}_0\)
and the left  side does not belong to \(\mathcal{I}_0,\) it
follows that \(\lambda t_\alpha=1\) for all \(\alpha\in\Phi.\)
Hence \(\varphi_{\mathcal{I},\mathcal{I}}(i_\alpha)=i_\alpha\) for
all \(\alpha\in\Phi.\)

Since \(\dim \mathcal{H}\geq 2,\)  any row of the Cartan matrix of a
simple Lie algebra \(\mathcal{G}\) contains negative number (see
\cite[page 59]{Hum}). This means that for any simple root
\(\alpha_i\in \Phi\) there exists a simple root \(\alpha_j\in
\Phi\) such that
\[
a_{i,j}=(\alpha_i, \alpha_j)<0,
\]
where \((\cdot, \cdot)\) is a bilinear form on
\(\mathcal{H}^\ast\) induced by the Killing form on
\(\mathcal{G}.\) Then by \cite[Page~45,~Lemma 9.4]{Hum}, we obtain
that \(\alpha_i+\alpha_j\) is also a root and \([e_{\alpha_j},
e_{\alpha_i}]=n_{\alpha_j,\alpha_i}e_{\alpha_i+\alpha_j},\) where
\(n_{\alpha_j,\alpha_i}\) is a non zero integer.

Further
\begin{eqnarray*}
[i_{\alpha_j}, e_{\alpha_i}] & = & [\theta(e_{\alpha_j}),
e_{\alpha_i}]=\theta\left([e_{\alpha_j},
e_{\alpha_i}]\right)=\theta(n_{\alpha_j,\alpha_i}e_{\alpha_i+\alpha_j})=\\
& = & n_{\alpha_j,\alpha_i}i_{\alpha_i+\alpha_j}.
\end{eqnarray*}
Applying to this equality \(\varphi,\) we obtain that
\begin{eqnarray*}
n_{\alpha_j,\alpha_i}i_{\alpha_i+\alpha_j}  & = &
n_{\alpha_j,\alpha_i}\varphi(i_{\alpha_i+\alpha_j})=\varphi([i_{\alpha_j},
e_{\alpha_i}])=[\varphi(i_{\alpha_j}),
\varphi(e_{\alpha_i})]=\\
&=& [i_{\alpha_j},t_{\alpha_i}
e_{\alpha_i}]=n_{\alpha_j,\alpha_i}t_{\alpha_i}i_{\alpha_i+\alpha_j}.
\end{eqnarray*}
Thus \(t_{\alpha_i}=1\) for all \(i=1,\ldots, k,\) i.e.
\(\varphi_{\mathcal{G},\mathcal{G}}\) acts identically on the
subset of all simple roots \(\{h_{\alpha_i}, e_{\alpha_i},
e_{-\alpha_i}: 1\leq i \leq  k\}.\) Since this subset generates
the algebra \(\mathcal{G},\) it follows that
\(\varphi_{\mathcal{G},\mathcal{G}}\) acts identically on
\(\mathcal{G},\) i.e.,
\(\varphi_{\mathcal{G},\mathcal{G}}=\textrm{id}_\mathcal{G}.\) By
Lemma~\ref{iden} there exists a number \(\lambda\) such that
\(\varphi_{\mathcal{I},\mathcal{I}}=\lambda\textrm{id}_\mathcal{I}.\)
Since \(\varphi_{\mathcal{I},\mathcal{I}}(i_\alpha)=i_\alpha,\) it
follows that \(\lambda=1,\) i.e.,
\(\varphi_{\mathcal{I},\mathcal{I}}=\textrm{id}_\mathcal{I}.\)

Finally
\begin{eqnarray*}
i_0 & = &
\varphi_{\mathcal{G},\mathcal{I}}(h_0)+\varphi_{\mathcal{I},\mathcal{I}}(i_0)=\varphi_{\mathcal{G},\mathcal{I}}(h_0)+i_0,
\end{eqnarray*}
implies that \(\varphi_{\mathcal{G},\mathcal{I}}(h_0)=0,\) and
therefore  \(\varphi_{\mathcal{G},\mathcal{I}}\equiv 0.\) So,
\[
\varphi=\varphi_{\mathcal{G},\mathcal{G}}+\varphi_{\mathcal{I},\mathcal{I}}=\textrm{id}_\mathcal{L}.
\]
The proof is complete.
\end{proof}

\begin{exam}\label{exam} Lemma \ref{lm57} is not true for algebras with \(\dim
\mathcal{H}=1.\)

There is a unique complex simple Leibniz algebra with
one-dimensional Cartan subalgebra and \(\dim
\mathcal{G}=\dim\mathcal{I}.\)  This is the 6-dimensional simple
Leibniz algebra
\[
\mathcal{L}=\mathfrak{sl}_2\dot{+}\textrm{span}\{x_0, x_1,
x_2\}=\textrm{span}\{h, e, f, x_0, x_1, x_2\},
\]
and non zero products of the basis vectors in $\mathcal{L}$ are
represented as follows \cite{ROT}:
$$\begin{array}{lll}
\, [x_k,e]=-k(3-k)x_{k-1}, & k\in \{1,2\},&\\
\, [x_k,f]=x_{k+1},  & k\in\{0, 1\}, & \\
\, [x_k,h]=(2-2k)x_k, & k\in\{0, 1,2\},&\\
\, [e,h]=2e, & [h,f]=2f, &[e,f]=h, \\
\, [h,e]=-2e,& [f,h]=-2f, & [f,e]=-h.\\
 \end{array}$$

Note that the \(\mathfrak{sl}_2\)-module isomorphism
\(\theta:\mathfrak{sl}_2\to \mathcal{I}\) is defined by
\begin{eqnarray*}
\theta(h)=2x_1,\, \theta(e)=2x_0,\, \theta(f)=x_2.
\end{eqnarray*}

Let  $\varphi\in Aut(\mathcal{L})$ be an automorphism such that
$\varphi(h_0+i_0)=h_0+i_0,$ where  \(h_0=h,\) \(i_0=x_1+x_0+x_2.\)
Then either \(\varphi=\textrm{id}_\mathcal{L},\) or either
\begin{eqnarray}\label{sl2}
\varphi=\varphi_{\mathcal{G},\mathcal{G}}+\theta\circ\varphi_{\mathcal{G},\mathcal{G}}-
\theta\circ\varphi_{\mathcal{G},\mathcal{G}}\circ \theta^{-1},
\end{eqnarray}
where
\begin{eqnarray*}
\varphi_{\mathcal{G},\mathcal{G}}(h)=h,\,
\varphi_{\mathcal{G},\mathcal{G}}(e)=-e,\,
\varphi_{\mathcal{G},\mathcal{G}}(f)=-f.
\end{eqnarray*}

Let \(\varphi=\varphi_{\mathcal{G},\mathcal{G}}+
\omega\theta\circ\varphi_{\mathcal{G},\mathcal{G}}+\lambda\theta\circ\varphi_{\mathcal{G},\mathcal{G}}\circ
\theta^{-1}.\) Then   $\varphi(h+i_0)=h+i_0$ implies that
$\varphi_{\mathcal{G},\mathcal{G}}(h)=h\) and
\(\omega\theta(h)+\lambda\theta(\varphi_{\mathcal{G},\mathcal{G}}(\theta^{-1}(i_0)))=i_0.$
Using Lemma~\ref{lm01} we obtain that
\begin{eqnarray*}
&\,&
\varphi_{\mathcal{G},\mathcal{G}}(h)=h,\,\varphi_{\mathcal{G},\mathcal{G}}(e)=t
e,\,\varphi_{\mathcal{G},\mathcal{G}}(f)=t^{-1} f.
\end{eqnarray*}
Since
\begin{eqnarray*}
x_1+x_0+x_2=i_0=\omega\theta(h)+\lambda\theta(\varphi_{\mathcal{G},\mathcal{G}}(\theta^{-1}(i_0)))=
(2\omega+\lambda)x_1+\lambda t x_0+\lambda t^{-1}x_2,
\end{eqnarray*}
it follows that \(2\omega+\lambda=1,\, \lambda t =\lambda
t^{-1}=1.\)

Case 1. \(\lambda=t=1.\) In this case \(\omega=0.\) Thus
\(\varphi_{\mathcal{G},\mathcal{G}}=\textrm{id}_\mathcal{G},\) and
therefore \(\theta\circ\varphi_{\mathcal{G},\mathcal{G}}\circ
\theta^{-1}=\textrm{id}_\mathcal{I}.\) Hence
\(\varphi=\textrm{id}_\mathcal{L}.\)

Case 2. \(\lambda=t=-1.\) In this case \(\omega=1,\) and we obtain
an automorphism of the  form~\eqref{sl2}.
\end{exam}

\begin{lm}\label{locsl2}
Let \(\nabla\) be a 2-local automorphism of the complex simple Leibniz
algebra $\mathcal{L}=\mathfrak{sl_2}\oplus\mathcal{I},\) where
\(\mathcal{I}=\textrm{span}\{x_0, x_1, x_2\},\) such that
\(\nabla(h)=h\) and \(\nabla(h+i_0)=h+i_0,\) where
\(i_0=x_0+x_1+x_2.\) Then \(\nabla=\textrm{id}_\mathcal{L}.\)
\end{lm}

\begin{proof}  Let \(x\in \mathfrak{sl}_2.\) Take an automorphism
\(\varphi^{x,h}=\varphi^{x,h}_{\mathfrak{sl}_2,\mathfrak{sl}_2}+\varphi^{x,h}_{\mathfrak{sl}_2,\mathcal{I}}
+\varphi^{x,h}_{\mathcal{I},\mathcal{I}}\in
\textrm{Aut}(\mathcal{L})\) such that
\(\nabla(x)=\varphi^{x,h}(x)\) and \(\nabla(h)=\varphi^{x,h}(h).\)
Since
\begin{eqnarray*}
h=\nabla(h)=\varphi^{x,h}(h)=\varphi^{x,h}_{\mathfrak{sl}_2,\mathfrak{sl}_2}(h)+\varphi^{x,h}_{\mathfrak{sl}_2,
\mathcal{I}}(h),
\end{eqnarray*}
it follows that \(\varphi^{x,h}_{\mathfrak{sl}_2,\mathcal{I}}\equiv
0.\) Then
\begin{eqnarray*}
\nabla(x)=\varphi^{x,h}(x)=\varphi^{x,h}_{\mathfrak{sl}_2,\mathfrak{sl}_2}(x)\in
\mathfrak{sl}_2, \ \forall \ x\in \mathfrak{sl}_2.
\end{eqnarray*}

Let now \(x\in \mathfrak{sl}_2\) be a non zero element. Take an
automorphism
\(\varphi^{x,h+i_0}=\varphi^{x,h+i_0}_{\mathfrak{sl}_2,\mathfrak{sl}_2}+
\varphi^{x,h+i_0}_{\mathfrak{sl}_2,\mathcal{I}}+\varphi^{x,h+i_0}_{\mathcal{I},\mathcal{I}}\)
such that \(\nabla(x)=\varphi^{x,h+i_0}(x)\) and
\(\nabla(h+i_0)=\varphi^{x,h+i_0}(h+i_0).\) Since \(\nabla(x)\in
\mathfrak{sl}_2\) and
\begin{eqnarray*}
\nabla(x)=\varphi^{x,h+i_0}(x)=\varphi^{x,h+i_0}_{\mathfrak{sl}_2,\mathfrak{sl}_2}(x)
+\varphi^{x,h+i_0}_{\mathfrak{sl}_2,\mathcal{I}}(x),
\end{eqnarray*}
we have that \(\varphi^{x,h+i_0}_{\mathfrak{sl}_2,\mathcal{I}}\equiv
0.\) Then
\begin{eqnarray*}
h+i_0=\nabla(h+i_0)=\varphi^{x,h+i_0}_{\mathfrak{sl}_2,\mathfrak{sl}_2}(h)+
\varphi^{x,h+i_0}_{\mathcal{I},\mathcal{I}}(i_0),
\end{eqnarray*}
implies that
\(\varphi^{x,h+i_0}_{\mathfrak{sl}_2,\mathfrak{sl}_2}(h)=h\) and
\(\varphi^{x,h+i_0}_{\mathcal{I},\mathcal{I}}(i_0)=i_0.\) By
Example~\ref{exam} we have that
\(\varphi^{x,h+i_0}\equiv\textrm{id}_{\mathcal{L}},\) and
therefore
\begin{eqnarray*}
\nabla(x)=\varphi^{x,h+i_0}(x)=x.
\end{eqnarray*}

Let \(x\in \mathfrak{sl}_2\) be a non zero element and let \(i\in
\mathcal{I}.\) Take an automorphism
\(\varphi^{x,x+i}=\varphi^{x,x+i}_{\mathfrak{sl}_2,\mathfrak{sl}_2}+
\varphi^{x,x+i}_{\mathfrak{sl}_2,\mathcal{I}}+\varphi^{x,x+i}_{\mathcal{I},\mathcal{I}}\)
such that \(\nabla(x)=\varphi^{x,x+i}(x)\) and
\(\nabla(x+i)=\varphi^{x,x+i}(x+i).\) Since
\begin{eqnarray*}
x=\nabla(x)=\varphi^{x,x+i}(x)=\varphi^{x,x+i}_{\mathfrak{sl}_2,\mathfrak{sl}_2}(x)+\varphi^{x,x+i}_{\mathfrak{sl}_2,\mathcal{I}}(x),
\end{eqnarray*}
it follows that \(\varphi^{x,x+i}_{\mathfrak{sl}_2,\mathcal{I}}\equiv
0.\) Then
\begin{eqnarray*}
\nabla(x+i)=\varphi^{x,x+i}_{\mathfrak{sl}_2,\mathfrak{sl}_2}(x)+\varphi^{x,x+i}_{\mathcal{I},\mathcal{I}}(i)=x+i',
\end{eqnarray*}
where \(i'\in \mathcal{I}.\)

Now we shall show that \(\nabla(x+i)=x+i\) for all \(x\in
\mathfrak{sl}_2,\) \(i\in \mathcal{I}.\)

Case 1. Let \(x=c_1 h+c_0 e +c_2 f,\) where \(|c_0|+|c_2|\neq 0,\)
and let \(i\in \mathcal{I}.\) Take an automorphism
\(\varphi^{x,h+i_0}=\varphi^{x,h+i_0}_{\mathfrak{sl}_2,\mathfrak{sl}_2}+
\varphi^{x,h+i_0}_{\mathfrak{sl}_2,\mathcal{I}}+\varphi^{x,h+i_0}_{\mathcal{I},\mathcal{I}}\)
such that \(\nabla(x)=\varphi^{x,h+i_0}(x)\) and
\(\nabla(h+i_0)=\varphi^{x,h+i_0}(h+i_0).\) Since
\(\varphi^{x,h+i_0}_{\mathfrak{sl}_2,\mathfrak{sl}_2}(h)=h,\) by
Lemma~\ref{lm01} there exists a non zero number \(t\) such that
\(\varphi^{x,h+i_0}_{\mathfrak{sl}_2,\mathfrak{sl}_2}(e)=t e\) and
\(\varphi^{x,h+i_0}_{\mathfrak{sl}_2,\mathfrak{sl}_2}(f)=t^{-1}
f.\) Then
\begin{eqnarray*}
x=\varphi^{x,h+i_0}_{\mathfrak{sl}_2,\mathfrak{sl}_2}(x)=c_1 h+t  c_0 e+
t^{-1} c_2 f.
\end{eqnarray*}
Since \(|c_0|+|c_2|\neq 0,\) it follows that \(t=1.\) Thus
\(\varphi^{x,h+i_0}_{\mathfrak{sl}_2,\mathfrak{sl}_2}=\textrm{id}_{\mathfrak{sl}_2}\)
and
\(\varphi^{x,h+i_0}_{\mathcal{I},\mathcal{I}}=\lambda\,\textrm{id}_\mathcal{I}\)
(see Lemma~\ref{iden}). Further
\begin{eqnarray*}
i_0=\varphi^{x,h+i_0}_{\mathfrak{sl}_2,\mathcal{I}}(h)+\varphi^{x,h+i_0}_{\mathcal{I},\mathcal{I}}(i_0)=\omega
\theta (h)+\lambda i_0=(2\omega+\lambda)x_1+\lambda(x_0+x_2).
\end{eqnarray*}
Thus \(\lambda=1\) and \(\omega=0,\) and therefore
\(\varphi^{x,h+i_0}=\textrm{id}_{\mathfrak{sl}_2}.\) So,
\begin{eqnarray*}
\nabla(x+i)=x+i.
\end{eqnarray*}

Case 2. Let \(x=c_1 h+c_0 e +c_2 f,\) where \(c_1\neq 0,\)  and
let \(i\in \mathcal{I}.\) Take an automorphism
\(\varphi^{x,e+i_0}=\varphi^{x,e+i_0}_{\mathfrak{sl}_2,\mathfrak{sl}_2}+
\varphi^{x,e+i_0}_{\mathfrak{sl}_2,\mathcal{I}}+\varphi^{x,e+i_0}_{\mathcal{I},\mathcal{I}}\)
such that \(\nabla(x)=\varphi^{x,e+i_0}(x)\) and
\(\nabla(e+i_0)=\varphi^{x,e+i_0}(e+i_0).\) From Case 1, it
follows that \(\nabla(e+i_0)=e+i_0.\) Then
\begin{center}
\(\varphi^{x,e+i_0}_{\mathfrak{sl}_2,\mathfrak{sl}_2}(h)=h\) and
\(\varphi^{x,e+i_0}_{\mathfrak{sl}_2,\mathfrak{sl}_2}(e)=e,\)
\end{center}
because \(c_1\neq0.\) Thus Lemma~\ref{lm01} implies that
\(\varphi^{x,e+i_0}_{\mathfrak{sl}_2,\mathfrak{sl}_2}\equiv\textrm{id}_{\mathfrak{sl}_2}.\)
By Lemma~~\ref{iden} we have
\(\varphi^{x,e+i_0}_{\mathcal{I},\mathcal{I}}=\lambda\textrm{id}_\mathcal{I}.\)
Further
\begin{eqnarray*}
i_0=\varphi^{x,e+i_0}_{\mathfrak{sl}_2,\mathcal{I}}(e)+\varphi^{x,e+i_0}_{\mathcal{I},\mathcal{I}}(i_0)=\omega
\theta (e)+\lambda i_0=(2\omega+\lambda)x_0+\lambda(x_1+x_2).
\end{eqnarray*}
Thus \(\lambda=1\) and \(\omega=0,\) and therefore
\(\varphi^{x,e+i_0}=\textrm{id}_{\mathfrak{sl}_2}.\) So,
\begin{eqnarray*}
\nabla(x+i)=x+i.
\end{eqnarray*}

Case 3. Let \(i\in \mathcal{I}.\)  Take an automorphism
\(\varphi^{i,h+i}\) such that \(\nabla(i)=\varphi^{i,h+i}(i)\) and
\(\nabla(h+i)=\varphi^{i,h+i}(h+i).\) From Case 2, it follows that
\[
h+i=\nabla(h+i)=\varphi^{i,h+i}_{\mathfrak{sl}_2,\mathfrak{sl}_2}(h)+
\varphi^{i,h+i}_{\mathfrak{sl}_2,\mathcal{I}}(h)+
\varphi^{i,h+i}_{\mathcal{I},\mathcal{I}}(i),
\]
and therefore
\[
i=\varphi^{i,h+i}_{\mathfrak{sl}_2,\mathcal{I}}(h)+
\varphi^{i,h+i}_{\mathcal{I},\mathcal{I}}(i).
\]
Then
\begin{eqnarray}\label{sss}
i-\nabla(i)=(\varphi^{i,h+i}_{\mathfrak{sl}_2,\mathcal{I}}(h)+
\varphi^{i,h+i}_{\mathcal{I},\mathcal{I}}(i))-\varphi^{i,h+i}_{\mathcal{I},\mathcal{I}}(i)=
\varphi^{i,h+i}_{\mathfrak{sl}_2,\mathcal{I}}(h)=\ast x_1.
\end{eqnarray}

Now take an automorphism \(\varphi^{i,e+i}\) such that
\begin{center}
\(\nabla(i)=\varphi^{i,e+i}(i)\) and
\(\nabla(e+i)=\varphi^{i,e+i}(e+i).\)
\end{center}
From Case 1, it follows that \(\nabla(e+i)=e+i.\) Then
\begin{eqnarray}\label{s12}
i-\nabla(i)=(\varphi^{i,e+i}_{\mathfrak{sl}_2,\mathcal{I}}(e)+
\varphi^{i,e+i}_{\mathcal{I},\mathcal{I}}(i))-\varphi^{i,e+i}_{\mathcal{I},\mathcal{I}}(i)=
\varphi^{i,e+i}_{\mathfrak{sl}_2,\mathcal{I}}(e)=\ast x_0.
\end{eqnarray}
Combining \eqref{sss} and \eqref{s12}, we obtain that
\(\nabla(i)=i.\) The proof is complete.
\end{proof}

\begin{thm}
Any 2-local automorphism of complex simple Leibniz algebra
$\mathcal{L}=\mathcal{G}\dot{+}\mathcal{I}$ is an automorphism.
\end{thm}

\begin{proof} Case 1. Let \(\dim \mathcal{G}\neq\dim\mathcal{I}\) or
\(\dim \mathcal{H}\geq 2.\) Let $\nabla$ be a 2-local automorphism
and $\nabla(h_0+i_0)=\varphi_{h_0+i_0}(h_0+i_0)$ for some
$\varphi_{h_0+i_0}\in Aut(\mathcal{L})$. Denote
$\widetilde{\nabla}=\varphi^{-1}_{h_0+i_0}\circ\nabla.$ Then for a
2-local automorphism $\widetilde{\nabla}$ we have
$\widetilde{\nabla}(h_0+i_0)=h_0+i_0$. For an element $x\in
\mathcal{L}$ there exists $\widetilde{\varphi}_{x, h_0+i_0}\in
Aut(\mathcal{L})$ such that
\begin{center}
\(\widetilde{\varphi}_{x,
h_0+i_0}(h_0+i_0)=\widetilde{\nabla}(h_0+i_0)=h_0+i_0\)
 and
$\widetilde{\nabla}(x)=\widetilde{\varphi}_{x, h_0+i_0}(x).$
\end{center}
Using the Lemmata \ref{lm02} and \ref{lm57} we conclude that
$\widetilde{\varphi}_{x, h_0+i_0}=\textrm{id}_\mathcal{L}$. Thus
$\widetilde{\nabla}(x)=\widetilde{\varphi}_{x, h_0+i_0}(x)=x$ for
each  $x\in \mathcal{L},$ and therefore
$\varphi^{-1}_{h_0+i_0}\circ\nabla=\textrm{id}_\mathcal{L}.$ Hence
$\nabla=\varphi_{h_0+i_0}$ is an automorphism.

Case 2. Let \(\mathcal{L}\) be an algebra from Example~\ref{exam}
and let \(\nabla\) be a 2-local automorphism on \(\mathcal{L}.\)
Take a 2-local automorphism  $\varphi_{h, h+i_0}$ such that
\(\nabla(h)=\varphi_{h, h+i_0}(h)\) and \linebreak
$\nabla(h+i_0)=\varphi_{h, h+i_0}(h+i_0).$ Then \(h\) and
\(h+i_0\) both are fixed points of 2-local automorphism
$\varphi^{-1}_{h, h+i_0}\circ\nabla,$ and therefore by
Lemma~~\ref{locsl2}, it is an identical automorphism. Thus
$\nabla=\varphi_{h, h+i_0}$ is an automorphism. The proof is
complete.
\end{proof}

\subsection{2-Local automorphisms on filiform  Leibniz
algebras}

\

The following theorems are similar to the corresponding theorems
for the Lie algebras case and their the proofs are obtained by
replacing the words "Lie algebra" \ by "Leibniz algebra" (see
\cite{AK2016A}).

\begin{thm}
Let $\mathcal{L}$ be an $n$-dimensional Leibniz algebra with
$n\geq2$. Suppose that
\begin{itemize}
\item[(i)] \(\dim [\mathcal{L},\mathcal{L}] \leq n-2;\)
\item[(ii)]\(\textrm{Ann}(\mathcal{L})\cap[\mathcal{L},\mathcal{L}]\neq \{0\}.\)
\end{itemize}
Then $\mathcal{L}$ admits a 2-local automorphism which is not an
automorphism.
\end{thm}

\begin{thm}
Let $\mathcal{L}$ be a finite-dimensional non null-filiform
nilpotent   Leibniz algebra    with $\dim \mathcal{L}\geq 2.$ Then
$\mathcal{L}$ admits a 2-local automorphism which is not an
automorphism.
\end{thm}

Let \(NF_n\) be the   unique $n$-dimensional nilpotent Leibniz with
condition $\dim[\mathcal{L},\mathcal{L}]=n-1$ (see the end of
Section 3).

Let $\varphi\in Aut(NF_n)$ and
$\varphi(e_1)=\sum\limits_{i=1}^{n}\alpha_i e_i$ for some
$\alpha_i\in \mathbb{C}, \ \alpha_1\neq0$, then it is easy to
check that
\begin{eqnarray*}
\varphi(e_j)=\alpha_1^{j-1}\sum_{i=1}^{n+1-j}\alpha_ie_{i+j-1}, \
2\leq j \leq n.
\end{eqnarray*}
Using  this property, as in the case of derivations, we  conclude
that an automorphism of $NF_n$ is uniquely defined by its value on the element
$e_1$ and any 2-local automorphism of this algebra  is an
automorphism.

\subsection{Local automorphisms on simple Leibniz algebras}

\

The following result shows that the problem concerning local automorphism of simple Leibniz algebras is reduced to the similar problem for simple Lie algebras.

\begin{thm}
Let \(\nabla\) be a local automorphism of complex simple Leibniz
algebra $\mathcal{L}=\mathcal{G}\dot{+}\mathcal{I}.$  Then
\(\nabla\) is an automorphism if and only if its
\(\nabla_{\mathcal{G},\mathcal{G}}\) part is an automorphism of the Lie algebra $\mathcal{G}$.
\end{thm}

\begin{proof} The  necessity part is evident and we shall consider the sufficient part.

Case 1. \(\dim \mathcal{G}=\dim \mathcal{I}.\) Take basis's
\(\{x_1, \ldots, x_m\}\) and \(\{y_i: y_i=\theta(x_i), i\in
\overline{1, m}\}\) on~\(\mathcal{G}\) and  \(\mathcal{I},\)
respectively, as in the proof of Lemma~\ref{GeqI}.

Suppose that \(\nabla\) is  a local automorphism of
\(\mathcal{L}\) such that its \(\nabla_{\mathcal{G},\mathcal{G}}\)
part is an automorphism. Consider an automorphism
\(\psi=\nabla_{\mathcal{G},\mathcal{G}}+ \theta\circ
\nabla_{\mathcal{G},\mathcal{G}}\circ \theta^{-1}.\) Then
\(\psi^{-1} \circ \nabla\) is a local automorphism of
\(\mathcal{L}\) such that \((\psi^{-1} \circ
\nabla)_{\mathcal{G},\mathcal{G}}=\textrm{id}_\mathcal{G}.\) So,
below it suffices to consider a local automorphism \(\nabla\) such
that \(\nabla_{\mathcal{G},\mathcal{G}}=\textrm{id}_\mathcal{G}.\)

Let \(x_k\in \mathcal{G}.\) Then
\begin{eqnarray*}
\nabla(x_k)&=&
\nabla_{\mathcal{G},\mathcal{G}}(x_k)+\nabla_{\mathcal{G},\mathcal{I}}(x_k)
=x_k+\nabla_{\mathcal{G},\mathcal{I}}(x_k).
\end{eqnarray*}
Take an automorphism \(\varphi^{x_k}\) such that
\(\nabla(x_k)=\varphi^{x_k}(x_k).\) Then
\begin{eqnarray*}
\nabla(x_k)&=&\varphi^{x_k}(x_k)=\varphi^{x_k}_{\mathcal{G},\mathcal{G}}(x_k)+\omega_{x_k}\theta(\varphi^{x_k}_{\mathcal{G},\mathcal{G}}(x_k)).
\end{eqnarray*}
Comparing the last two equalities we obtain that
\(\varphi^{x_k}_{\mathcal{G},\mathcal{G}}(x_k)=x_k,\) and
therefore
\begin{eqnarray*}
\nabla(x_k)&=&\varphi^{x_k}_{\mathcal{G},\mathcal{G}}(x_k)+\omega_{x_k}\theta(\varphi^{x_k}_{\mathcal{G},\mathcal{G}}(x_k))=\\
&=& x_k+\omega_{x_k}\theta(x_k)=x_k+\omega_{x_k}y_k.
\end{eqnarray*}

Likewise  for an element  \(x=x_k+x_s\in \mathcal{G}\) we have
that
\begin{eqnarray*}
\nabla(x)&=&\varphi^{x}(x)=\varphi^{x}_{\mathcal{G},\mathcal{G}}(x_k+x_s)+
\omega_{x_k+x_s}\theta(\varphi^{x}_{\mathcal{G},\mathcal{G}}(x_k+x_s))=\\
&=&
x_k+x_s+\omega_{x_k+x_s}\theta(x_k+x_s)=x_k+x_s+\omega_{x_k+x_s}(y_k+y_s).
\end{eqnarray*}
Since
\begin{eqnarray*}
\nabla(x) &=&
\nabla(x_k+x_s)=\nabla(x_k)+\nabla(x_s)=x_k+x_s+\omega_{x_k}y_k+\omega_{x_s}y_s,
\end{eqnarray*}
we have that \(\omega_{x_k+x_s}=\omega_{x_k}=\omega_{x_s}.\) This
means that there exists  \(\omega\in \mathbb{C}\) such that
\(\nabla(x)=x+\omega \theta(x),\) i.e.,
\(\nabla_{\mathcal{G},\mathcal{I}}=\omega\theta.\)

Now take an element  \(x=x_k+y_k\in \mathcal{G}+\mathcal{I}\) and
an automorphism \(\varphi^{x}\) such that
\(\nabla(x)=\varphi^{x}(x).\) Then
\begin{eqnarray*}
\nabla(x)&=&\varphi^{x}(x)=\varphi^{x}_{\mathcal{G},\mathcal{G}}(x_k)+
\omega_{x_k}\theta(\varphi^{x}_{\mathcal{G},\mathcal{G}}(x_k))+\lambda_{x}\theta(\varphi^{x}_{\mathcal{G},\mathcal{G}}(\theta^{-1}(y_k)))=\\
&=&
x_k+\omega_{x}\theta(x_k)+\lambda_{x}y_k=x_k+(\omega_{x}+\lambda_x)y_k.
\end{eqnarray*}

Further for an element \(x=x_k+x_s+y_k+y_s\in
\mathcal{G}+\mathcal{I}\) take an automorphism \(\varphi^{x}\)
such that \(\nabla(x)=\varphi^{x}(x).\) Then
\begin{eqnarray*}
\nabla(x)&=&\varphi^{x}(x)=\varphi^{x}_{\mathcal{G},\mathcal{G}}(x_k+x_s)+
\omega_{x_k}\theta(\varphi^{x}_{\mathcal{G},\mathcal{G}}(x_k+x_s))+
\lambda_{x}\theta(\varphi^{x}_{\mathcal{G},\mathcal{G}}(\theta^{-1}(y_k+y_s)))=\\
&=&
x_k+x_s+\omega_{x}\theta(x_k+x_s)+\lambda_{x}(y_k+y_s)=x_k+x_s+(\omega_{x}+\lambda_x)(y_k+y_s).
\end{eqnarray*}
Combing the last two equalities we have  that there exists \(t\in
\mathbb{C}\) such that \(\nabla(x_k+y_s)=x_k+ty_k.\) Then
\begin{eqnarray*}
\nabla(y_k)&=&\nabla(x_k+y_k)-\nabla(x_k)=\\
&=& x_k+ty_k-(x_k+\omega y_k)=(t-\omega)y_k.
\end{eqnarray*}
This means that
\(\nabla_{\mathcal{I},\mathcal{I}}=(t-\omega)\textrm{id}_\mathcal{I}.\)
Hence
\(\nabla=\textrm{id}_\mathcal{G}+\omega\theta+(t-\omega)\textrm{id}_\mathcal{I}.\)

Case 2. \(\dim \mathcal{G}\neq\dim \mathcal{I}.\) Take an element
of the form \eqref{inol}, i.e.,
\begin{eqnarray*}
i_0=\sum_{\beta\in \Gamma}\sum_{k=1}^{s_\beta}i^{(k)}_{l_\beta}.
\end{eqnarray*}

Let   \(\varphi_{h_0+i_0}\) be an automorphism such that
\(\nabla(h_0+i_0)=\varphi_{h_0+i_0}(h_0+i_0).\) If necessary, we
can replace \(\nabla\) by \(\varphi^{-1}_{h_0+i_0}\circ\nabla,\)
and suppose that \(\nabla(h_0+i_0)=h_0+i_0.\)

Since
\[
h_0+i_0=\nabla(h_0+i_0)=\nabla_{\mathcal{G},\mathcal{G}}(h_0)+\nabla_{\mathcal{I},\mathcal{I}}(i_0),
\]
it follows that $\nabla_{\mathcal{G},\mathcal{G}}(h_0)=h_0$ and
$\nabla_{\mathcal{I},\mathcal{I}}(i_0)=i_0.$ Since
\(\nabla_{\mathcal{G},\mathcal{G}}\) is an automorphism, by
Lemma~\ref{lm01}, for every \(\alpha\in \Phi\) there exists a non
zero \(t_\alpha\in \mathbb{C}\) such that
\(\nabla_{\mathcal{G},\mathcal{G}}(e_\alpha)=t_\alpha e_\alpha,\)
\(\nabla_{\mathcal{G},\mathcal{G}}(e_{-\alpha})=t^{-1}_\alpha
e_{-\alpha}\)  and \(\nabla_{\mathcal{G},\mathcal{G}}(h)=h\) for
all \(h\in \mathcal{H}.\)

Let \(i_\beta\in \mathcal{I}_\beta.\) Then
\begin{eqnarray*}
\nabla(h_0+i_\beta)&=& \nabla(h_0)+\nabla(i_\beta)
=h_0+\nabla(i_\beta).
\end{eqnarray*}
Take an automorphism \(\varphi^{h_0+i_\beta}\) such that
\(\nabla(h_0+i_\beta)=\varphi^{h_0+i_\beta}(h_0+i_\beta).\) Then
\begin{eqnarray*}
\nabla(h_0+i_\beta)&=&\varphi^{h_0+i_\beta}(h_0+i_\beta)=\varphi^{h_0}_{\mathcal{G},\mathcal{G}}(h_0)+
\varphi^{h_0+i_\beta}_{\mathcal{I},\mathcal{I}}(i_\beta).
\end{eqnarray*}
Comparing the last two equalities we obtain that
\(\varphi^{h_0+i_\beta}_{\mathcal{G},\mathcal{G}}(h_0)=h_0,\) and
therefore \(\varphi^{h_0+i_\beta}\) acts as diagonal matrix on
\(\mathcal{L}.\) Thus
\begin{eqnarray*}
\nabla(i_\beta)&=&\varphi^{h_0+i_\beta}_{\mathcal{I},\mathcal{I}}(i_\beta)=c_\beta
i_\beta.
\end{eqnarray*}
Since \(\nabla(i_0)=i_0,\) it follows that
\(\nabla(i_\beta)=i_\beta\) for all \(\beta.\) So,
\(\nabla_{\mathcal{I},\mathcal{I}}=\textrm{id}_\mathcal{I}.\)

Let \(\alpha\in \Phi.\)  Considering \(\mathcal{I}\) as
\((\mathfrak{sl_2})_\alpha\)-module,  where
\((\mathfrak{sl_2})_\alpha\equiv\textrm{span}\{e_\alpha,
e_{-\alpha},  h_\alpha=[e_{\alpha},  e_{-\alpha}]\},\) we can find a non
trivial irreducible submodule \(\mathcal{J}_\alpha\) of
\(\mathcal{I}.\) Then \(\mathcal{J}_\alpha\) admits a basis
\(\{x_0^{\alpha}, \ldots, x_n^{\alpha}\}\) such that \cite{ROT}:
$$\begin{array}{lll}
\, [x_k^{\alpha}, e_\alpha]=x_{k+1}^{\alpha},  & k\in\{0, \ldots, n-1\}, &  \\
\, [x_k^{\alpha}, e_{-\alpha}]=- k(n + 1 - k)x_{k-1}^{\alpha}, & k\in \{1,\ldots, n\},&\\
\, [x_k^{\alpha}, h_\alpha]=(n-2k)x_k^{\alpha}, & k\in\{0,\ldots, n\}.&\\
\end{array}
$$
The matrix of the right multiplication operator
\(R_{h_\alpha+e_\alpha}\) on \(\mathcal{J}_\alpha\) has the
following form
\[
\left(%
\begin{array}{ccccccc}
  n & 0 & 0 &  \cdots & 0 & 0 & 0\\
  1 & n-2 & 0  &  \cdots & 0 & 0 & 0\\
  0 & 1 & n-4 &  \cdots & 0 & 0 & 0\\
  \vdots &  \vdots & \vdots & \vdots & \vdots & \vdots & \vdots\\
  0 & 0 & 0 &  1 & -(n-4) & 0 & 0 \\
  0 & 0 & 0 &  0 & 1      & -(n-2) & 0 \\
0 & 0 & 0 &  0 & 0 & 1 & -n \\
\end{array}%
\right).
\]
Direct computations show that \(n\) is a eigenvalue of this matrix
and we take a non zero eigenvector \(i=\sum\limits_{s=0}^n t_s
x_s^{\alpha} \in \mathcal{J}_\alpha\) corresponding to this
eigenvalue, i.e., \([i, h_\alpha+e_\alpha]=n i,\) with \(t_0\neq
0.\) For an element \(x=h_\alpha+e_\alpha+i\) choose an
automorphism \(\varphi^x\) such that \(\nabla(x)=\varphi^x(x).\)
Then
\begin{eqnarray*}
[\nabla(x), \nabla(x)]&=& [h_\alpha+t_\alpha e_\alpha+i, h+t_\alpha
e_\alpha+i]=[i, h_\alpha+t_\alpha e_\alpha]
\end{eqnarray*}
and
\begin{eqnarray*}
[\nabla(x), \nabla(x)]&=& [\varphi^x(h_\alpha+e_\alpha+i),
\varphi^x(h_\alpha+e_\alpha+i)]=\\
&=& \varphi^x([h_\alpha+e_\alpha+i, h_\alpha+e_\alpha+i])=\\
&=& \varphi^x([i, h_\alpha+e_\alpha])=n\varphi^x(i)=n \nabla(i)=n
i.
\end{eqnarray*}
The last two equalities implies that
\[
 [i, h_\alpha+t_\alpha e_\alpha]=n i= [i, h_\alpha+e_\alpha].
\]
Comparing coefficients at the basis element \(x^\alpha_1\) in the
above equality we conclude that
\[
t_\alpha t_0+(n-2)t_1=t_0+(n-2)t_1.
\] Thus \(t_\alpha=1,\) and
therefore \(\nabla(e_\alpha)=e_{\alpha}.\) So,
\(\nabla=\textrm{id}_\mathcal{L}.\) The proof is complete.
\end{proof}

\section*{Acknowledgements}

The authors are grateful to professor Kaiming Zhao  for useful discussions and
suggestions.

\end{document}